\documentclass[12pt, reqno]{amsart}

\usepackage{amsmath}
\usepackage{amssymb,amsfonts,amscd}
\usepackage{latexsym}
\usepackage{mathrsfs}

\usepackage[mathscr]{euscript}

\makeatletter
\def\@settitle{\begin{center}%
    \baselineskip14\p@\relax
    \bfseries
    \@title
  \end{center}%
}
\makeatother

\usepackage{hyperref}

\newcommand\al{\alpha}

\newcommand\Ga{\Gamma}
\newcommand\de{\delta}

\newcommand\la{\lambda}

\newcommand{\eps}{\varepsilon}
\newcommand\ka{\varkappa}
\newcommand\si{\sigma}

\newcommand\C{\mathbb C}

\newcommand\Z{\mathbb Z}
\newcommand\Y{\mathbb Y}
\newcommand\Q{\mathbb Q}

\newcommand\FF{\mathbb F}

\newcommand\A{\mathscr F}
\newcommand\B{\mathscr G}

\newcommand\HL{{\operatorname{HL}}}
\newcommand\W{{\operatorname{W}}}
\newcommand\Jack{{\operatorname{Jack}}}

\newcommand\sgn{\operatorname{sgn}}

\newcommand\const{\operatorname{const}}

\newcommand\Sym{\operatorname{Sym}}
\newcommand\Symd{\widehat{\Sym}_\FF}
\newcommand\Tab{\operatorname{Tab}}
\newcommand\RTab{\operatorname{RTab}}

\newcommand\wt{\widetilde}
\newcommand\wh{\widehat}
\newcommand\one{\mathbf1}
\newcommand\ccdot{\,\cdot\,}

\newcommand\Id{H}
\newcommand\Idm{\wt H}

\newcommand\Wm{\mathcal W}
\newcommand\spec{\eta}

\newcommand\kk{\mathbf k}

\newtheorem{theorem}{Theorem}[section]
\newtheorem{proposition}[theorem]{Proposition}
\newtheorem{lemma}[theorem]{Lemma}
\newtheorem{corollary}[theorem]{Corollary}

\theoremstyle{definition}
\newtheorem{definition}[theorem]{Definition}
\newtheorem{remark}[theorem]{Remark}
\newtheorem{example}[theorem]{Example}

\numberwithin{equation}{section}

\begin{document}

\title[]{\large Interpolation Macdonald polynomials and\\ Cauchy-type identities}

\author{Grigori Olshanski}
\address{Institute for Information Transmission Problems of the Russian Academy of Sciences, Bolshoy Karetny 19, Moscow 127051, Russia; \newline \indent Skolkovo Institute of Science and Technology, Nobel Street 3, Moscow 121205, Russia; \newline \indent National Research University Higher School of Economics, Myasnitskaya 20, Moscow 101000, Russia}
\email{olsh2007@gmail.com}

\begin{abstract}

Let $\Sym$ denote the algebra of symmetric functions and $P_\mu(\ccdot;q,t)$ and $Q_\mu(\ccdot;q,t)$ be the Macdonald symmetric functions (recall that they differ by scalar factors only). The $(q,t)$-Cauchy identity 
$$
\sum_\mu P_\mu(x_1,x_2,\dots;q,t)Q_\mu(y_1,y_2,\dots;q,t)=\prod_{i,j=1}^\infty\frac{(x_iy_jt;q)_\infty}{(x_iy_j;q)_\infty}
$$
expresses the fact that the $P_\mu(\ccdot;q,t)$'s form an orthogonal basis in $\Sym$ with respect to a special scalar product $\langle\ccdot,\ccdot\rangle_{q,t}$. The present paper deals with the inhomogeneous \emph{interpolation} Macdonald symmetric functions 
$$
I_\mu(x_1,x_2,\dots;q,t)=P_\mu(x_1,x_2,\dots;q,t)+\text{lower degree terms}.
$$ 
These functions come from the $N$-variate interpolation Macdonald  polynomials, extensively studied in the 90's by Knop, Okounkov, and Sahi. The goal of the paper is to construct symmetric functions $H_\mu(\ccdot;q,t)$ with the biorthogonality property 
$$
\langle I_\mu(\ccdot;q,t),  H_\nu(\ccdot;q,t)\rangle_{q,t}=\de_{\mu\nu}.
$$
These new functions live in a natural completion $\wh\Sym\supset\Sym$. As a corollary one obtains a new Cauchy-type identity in which the interpolation Macdonald polynomials are paired with certain multivariate rational symmetric functions. The degeneration of this identity in the Jack limit $(q,t)=(q,q^\kk)\to(1,1)$  is also described. 

\end{abstract}


\maketitle

\tableofcontents

\section{Introduction}\label{sect1}

\subsection{Preface}
One of the fundamental formulas in the theory of symmetric functions is the Cauchy identity
$$
\sum_{\mu\in\Y}s_\mu(x_1,x_2,\dots)s_\mu(y_1,y_2,\dots)=\prod_{i=1}^\infty\prod_{j=1}^\infty\frac1{1-x_iy_j}.
$$
Here $\Y$ is the set of partitions and $s_\mu(\ccdot)$ denotes the Schur symmetric function with index $\mu\in\Y$. The elements $s_\mu(\ccdot)$ form a distinguished orthonormal basis of the algebra of symmetric functions with respect to the canonical scalar product $\langle\ccdot,\ccdot\rangle$, and the right-hand side of the Cauchy identity is the reproducing kernel for this scalar product. 

Macdonald's $(q,t)$-deformation of the Cauchy identity \cite[Ch. VI, (4.13)]{M} has the form 
\begin{equation}\label{eq1.A}
\sum_{\mu\in\Y}P_\mu(x_1,x_2,\dots;q,t)Q_\mu(y_1,y_2,\dots;q,t)=\prod_{i=1}^\infty\prod_{j=1}^\infty\frac{(x_iy_jt;q)_\infty}{(x_iy_j;q)_\infty},
\end{equation}
where the $P_\mu(\ccdot;q,t)$ are the Macdonald symmetric functions with parameters $(q,t)$, $Q_\mu(\ccdot;q,t)$ differs from $P_\mu(\ccdot;q,t)$ by a scalar factor, and 
$$
(z;q)_\infty:=\prod_{n=0}^\infty(1-zq^n)
$$ 
is the conventional notation for the $q$-Pochhammer symbol. The elements $P_\mu(\ccdot;q,t)$ form an orthogonal basis with respect to the $(q,t)$-deformed scalar product $\langle\ccdot,\ccdot\rangle_{q,t}$ (see \cite[Ch. VI, (1.5)]{M}), and 
$$
Q_\mu(\ccdot;q,t)=\frac{P_\mu(\ccdot;q,t)}{\langle P_\mu(\ccdot;q,t),P_\mu(\ccdot;q,t)\rangle_{q,t}},
$$
so that
$$
\langle P_\mu(\ccdot;q,t),Q_\nu(\ccdot;q,t)\rangle_{q,t}=\de_{\mu\nu}.
$$ 

The basis $\{P_\mu(\ccdot;q,t)\}_{\mu\in\Y}$ is homogeneous; associated to it is an inhomogeneous basis whose elements 
$$
I_\mu(x_1,x_2,\dots;q,t)=P_\mu(x_1,x_2\dots;q,t)+\text{lower degree terms}
$$
are the \emph{$(q,t)$-interpolation symmetric functions}. These functions are built from the $N$-variate interpolation symmetric polynomials, which were studied in the works of Knop \cite{Knop}, Okounkov \cite{Ok-MRL}, \cite{Ok-CM}, \cite{Ok-AAM}, \cite{Ok-FAA}, and Sahi \cite{Sahi2}. 

The purpose of the present paper is to study the \emph{dual functions} 
$$
\Id_\mu(y_1,y_2,\dots;q,t)=Q_\mu(y_1,y_2,\dots;q,t)+\text{higher degree terms},
$$
which form, together with the interpolation symmetric functions $I_\mu(\ccdot;q,t)$, a biorthogonal system. That is, 
$$
\langle I_\mu(\ccdot;q,t),\Id_\nu(\ccdot;q,t)\rangle_{q,t}=\de_{\mu\nu}.
$$  

The biorthogonality property is equivalent to the Cauchy-type identity
\begin{equation}\label{eq1.C}
\sum_{\mu\in\Y}I_\mu(x_1,x_2,\dots;q,t)\Id_\mu(y_1,y_2,\dots;q,t)=\prod_{i=1}^\infty\prod_{j=1}^\infty\frac{(x_iy_jt;q)_\infty}{(x_iy_j;q)_\infty},
\end{equation}
with the same right-hand side as in \eqref{eq1.A}.

The functions $H_\mu(\ccdot; q,t)$ do not fit into the algebra of symmetric functions;  they lie in the completion of this algebra consisting of arbitrary symmetric formal power series in $y_1,y_2,\dots$, not necessarily of bounded degree.  

Sometimes it is convenient to work with the \emph{modified dual functions}
$$
\Idm_\mu(u_1,u_2,\dots;q,t):=\Id_\mu(u^{-1}_1,u^{-1}_2,\dots;q,t)\prod_{j=1}^\infty(u_j^{-1};q)_\infty,
$$ 
which are formal power series in $u_1^{-1}, u_2^{-1},\dots$\,. We also introduce the truncated functions
$$
\Idm_{\mu\mid N}(u_1,\dots,u_N;q,t):=\Idm_\mu(u_1,\dots,u_N, \infty,\infty,\dots; q,t),  \qquad \mu\in\Y(N).
$$ 
They turn out to be rational functions, and it was for this purpose that the extra factors $(u_j^{-1};q)_\infty$ have been added. 

The passage to power series in inverse powers of variables is rather a matter of taste; we do it in order to follow the notation of \cite{OO-AA} and \cite{OO-IMRN}.

\subsection{The results}

Here is a brief formulation of the main results. 

\smallskip

(1)  Theorem \ref{thm5.A} gives an explicit  \emph{combinatorial formula} for the dual functions:  $\Idm_\mu(u_1,u_2,\dots;q,t)$ is written as a weighted sum over the semistandard tableaux of shape $\mu$. As a corollary we obtain the fact mentioned above:  the truncated functions $\Idm_{\mu\mid N}(u_1,\dots,u_N;q,t)$ are rational  (Corollary \ref{cor5.B}).  

\smallskip
 
(2) Theorem \ref{thm7.A} shows that the truncated functions $\Idm_{\mu\mid N}(u_1,\dots,u_N;q,t)$ satisfy a system of $q$-difference equations. Namely, we exhibit a collection $D^1_N(q,t)$, \dots, $D^N_N(q,t)$ of  \emph{commuting partial $q$-difference operators} in variables $u_1,\dots,u_N$, such that the $\Idm_{\mu\mid N}(u_1,\dots,u_N;q,t)$ are their common eigenvectors:
$$
\left(1+\sum_{k=1}^N D^k_N(q,t)z^k\right)\Idm_{\mu\mid N}(u_1,\dots,u_N;q,t)=\prod_{i=1}^N(1+q^{\mu_i}t^{1-i}z)\cdot \Idm_{\mu\mid N}(u_1,\dots,u_N;q,t).
$$
This provides a characterization of the functions $\Idm_{\mu\mid N}(u_1,\dots,u_N;q,t)$. 
\smallskip
 
(3) Theorem \ref{thm8.A} gives a Cauchy-type identity for the \emph{interpolation Jack polynomials}. At a heuristical level, this identity can be derived from \eqref{eq1.C}  via a limit transition as $(q,t)\to(1,1)$ along the curve $t = q^\kk$, where $\kk$ is the Jack parameter. However, the justification of this procedure presents difficulties and we give another proof, based on difference equations.

\subsection{Examples and comments} 

A simple consequence of the combinatorial formula of Theorem \ref{thm5.A} is an explicit expression for the functions $\Idm_{\mu\mid N}(u_1,\dots,u_N;q,t)$ in the particular case $N=1$ (then we have $\mu=(m)$, where $m=1,2,\dots$):
$$
\Idm_{(m)\mid 1}(u;q,t)=\frac{(t;q)_m}{(q;q)_m}\,\frac1{(u-q^{-1})\dots(u-q^{-m})}.
$$

As a particular case of the Cauchy identity \eqref{eq1.C} one obtains a generating function for the one-row interpolation functions,
\begin{equation}\label{eq2.D1}
1+\sum_{m=1}^\infty \frac{(t;q)_m}{(q;q)_m}\frac{I_{(m)}(x_1,x_2,\dots;q,t)}{(u-q^{-1})\dots(u-q^{-m})}=(u^{-1};q)_\infty\prod_{i=1}^\infty\frac{(x_iu^{-1}t;q)_\infty}{(x_iu^{-1};q)_\infty},
\end{equation}
which first appeared, in a different form,  in Okounkov \cite[(2.10)]{Ok-MRL} (see also Proposition \ref{prop2.B} below). Its Jack analogue has the form 
\begin{equation}\label{eq1.D2}
1+\sum_{m=1}^\infty \frac{Q^*_{(m)}(x_1,x_2,\dots;\kk)}{u(u-1)\dots(u-m+1)}=\prod_{i=1}^\infty\frac{\Ga(x_i-u-\kk i)}{\Ga(x_i-u-\kk i+\kk)}\frac{\Ga(-u-\kk i+\kk)}{\Ga(-u-\kk i)},
\end{equation}
where $Q^*_{(m)}(\ccdot;\kk)$ is the shifted interpolation analogue of the one-row Jack symmetric function $Q_{(m)}(\ccdot;\kk)$ (see \cite[(2.10)]{OO-IMRN}). In the case $\kk=1$ the latter formula turns into
\begin{equation}\label{eq1.D3}
1+\sum_{m=1}^\infty\frac{h^*_m(x_1,x_2,\dots)}{u(u-1)\dots(u-m+1)}=\prod_{i=1}^\infty\frac{u+i}{u-x_i+i},
\end{equation}
where the $h^*_m$ are the one-row shifted Schur functions (see \cite[Theorem 12.1]{OO-AA}). 

In \cite{Ols-Notes}, I derived an extension of formula \eqref{eq1.D3} to the case of many variables $u_1, u_2,\dots$ providing a Cauchy-type identity (it was reproduced and exploited in Borodin--Olshanski \cite[Proposition 4.1]{BO-AM}).  Then much more general formulas, referring to multiparameter Schur polynomials, were independently found by Molev \cite{Molev}. These results  have served as motivation for the present work: I would like to understand what kind of non-polynomial symmetric functions may arise from Cauchy-type identities with two Macdonald's parameters $(q,t)$. 

A different family of rational symmetric functions earlier appeared in Borodin--Olshanski \cite{BO-AM} in the context of asymptotic theory of characters. Yet another family was recently introduced by Borodin \cite{B-AM} in connection with integrable lattice models;  more general functions are presented in Borodin--Petrov \cite{BP-SM}, \cite{BP-Lect}. Borodin's symmetric functions are a rational deformation of the Hall--Littlewood polynomials and have a number of nice properties including the existence of a dual set of rational functions and a Cauchy-type identity. 

However, it is not clear for me if all these examples can be united in the framework of some general theory. 

There are links between the present paper and the work of Cuenca \cite{Cuenca}. The main result of \cite{Cuenca} is an explicit realization of an infinite family of commuting operators, which act on the algebra of symmetric functions and are diagonalized in the basis $\{I_\mu(\ccdot;q,t)\}$. These operators are large-$N$ limits of (slightly renormalized) $q$-difference operators considered below in Section \ref{sect7}. They are analogues of  Nazarov--Sklyanin's ``Macdonald operators at infinity'' \cite{NS-2014}, which are diagonalized in the homogeneous basis $\{P_\mu(\ccdot;q,t)\}$. For the latter operators, Nazarov and Sklyanin derived a nice expression in terms of the Hall-Littlewood symmetric functions. The formula obtained by Cuenca  has a similar structure, but it involves certain inhomogeneous versions of Hall--Littlewood functions. These inhomogeneous functions turn out to be the limiting functions $\lim\limits_{q\to0}I_\la(\ccdot; q^{-1},t^{-1})$  discussed in Section \ref{sect9} below. 

Finally, note that Cauchy identities related to Okounkov's $BC$-type interpolation polynomials \cite{Ok-TG-1998} are derived in Rains \cite[section 6]{Rains}. However, in type $BC$ the picture is much more complicated than in type $A$.

\subsection{Organization of the paper}
In Section \ref{sect2} we formulate a few necessary results from the theory of $N$-variate interpolation polynomials with Macdonald's parameters $(q,t)$, developed in the works of Knop \cite{Knop}, Okounkov \cite{Ok-MRL}, \cite{Ok-CM}, \cite{Ok-AAM}, \cite{Ok-FAA}, and Sahi \cite{Sahi2}.  Then we explain how to extend these results to the case of infinitely many variables. 

In Section \ref{sect3} we introduce the dual functions, and in Section \ref{sect4} we discuss the special case of equal parameters $t=q$ when the Macdonald symmetric functions reduce to the Schur functions.

In Section \ref{sect5} we derive the combinatorial formula for the dual functions. The key computation is based on an adaptation of an argument from Rains' paper \cite{Rains}. 

In Sections \ref{sect6} and \ref{sect7}, we find analogues of Macdonald's $q$-difference operators, which are related to the dual functions. We examine first the special case $t=q$ and then pass to the general case using a trick from Macdonald's book.  

In Section \ref{sect8} we deal with the Jack degeneration and prove Theorem \ref{thm8.A} mentioned above (a version of Cauchy identity for the interpolation Jack polynomials). 

Section \ref{sect9} is a complement: here we consider two degenerations which appear in the limits  $t\to0$ (the $q$-Whittaker limit) and  $q\to0$ (the Hall--Littlewood limit).

\subsection{Acknowledgments}
I am grateful to Alexei Borodin and Cesar Cuenca for valuable comments. 

\section{Interpolation polynomials and functions}\label{sect2}

\subsection{Preliminaries}

Unless otherwise stated, we are working over the base field $\FF:=\Q(q,t)$,  
where $q$ and $t$ are treated as formal variables. Let us fix some notation. 

$\bullet$ $\Sym$ is the algebra of symmetric functions over $\Q$ and $\Sym_\FF:=\Sym\otimes_\Q\FF$.

$\bullet$ $\Y$ is the set of all partitions, which are identified with the corresponding Young diagrams. The length of a partition $\mu\in\Y$ (that is, the number of its nonzero parts) is denoted by $\ell(\mu)$. The subset $\Y(N)\subset\Y$ consists of the partitions with $\ell(\ccdot)\le N$.

$\bullet$ $P_\mu(x_1,x_2,\dots;q,t)$ is the Macdonald symmetric function with index $\mu\in\Y$. We often suppress the arguments and write $P_\mu(\ccdot;q,t)$. 

$\bullet$ $\Sym(N)$ and $\Sym_\FF(N)$ are the algebras of symmetric polynomials in $N$ variables over $\Q$ and $\FF$, respectively.

$\bullet$ $P_{\mu\mid N}(x_1,\dots,x_N;q,t)$ or $P_{\mu\mid N}(\ccdot;q,t)$ is the Macdonald symmetric polynomial; here the index $\mu$ is assumed to be in $\Y(N)$, otherwise the polynomial is zero. With this convention,
$$
P_{\mu\mid N}(x_1,\dots,x_N;q,t)=P_\mu(x_1,\dots,x_N,0,0,\dots;q,t).
$$

\begin{remark}\label{Pochhammer}
The infinite product 
$$
(z;q)_\infty:=\prod_{n=0}^\infty(1-zq^n)
$$
is initially defined as a formal power series in two indeterminates $z$ and $q$ (Macdonald \cite[ch. VI, \S2]{M}). However, due to the identity
$$
\prod_{n=0}^\infty(1-zq^n)=\exp\left(\log\left(\prod_{n=0}^\infty(1-zq^n)\right)\right)=\exp\left(-\sum_{m=1}^\infty\frac1m\,\frac{z^m}{1-q^m}\right)
$$
we may (and do!) treat $(z;q)_\infty$ as a power series in $z$ with coefficients in $\Q(q)$. This convention is used throughout the paper. In particular, one has the following identity (which will be used in Section \ref{sect9})
\begin{equation}\label{inversion}
(z;q^{-1})_\infty=\frac1{(zq;q)_\infty}.
\end{equation}
Indeed, 
$$
(z;q^{-1})_\infty=\exp\left(-\sum_{m=1}^\infty\frac1m\,\frac{z^m}{1-q^{-m}}\right)=\exp\left(\sum_{m=1}^\infty\frac1m\,\frac{(zq)^m}{1-q^m}\right)=\frac1{(zq;q)_\infty}.
$$
\end{remark}

\subsection{Interpolation Macdonald polynomials}
Here we rewrite in our notation a number of results due to Knop, Okounkov, and Sahi. 

Fix $N\in\{1,2,\dots\}$. Given $\la\in\Y(N)$, we set
$$
X_N(\la):=(q^{-\la_1}, q^{-\la_2}t, \dots, q^{-\la_N}t^{N-1})\in\FF^N,
$$
that is, the $i$th coordinate of $X_N(\la)$ equals $q^{-\la_i}t^{i-1}$. The vectors $X_N(\la)$ are regarded as \emph{interpolation nodes}.

\begin{proposition}\label{prop2.C}
For every $\mu\in\Y(N)$ there exists a unique, within a scalar factor, element $I_{\mu\mid N}(\ccdot;q,t)\in\Sym_\FF(N)$ such that{\rm:}

{\rm(1)} $\deg I_{\mu\mid N}(\ccdot;q,t)=|\mu|$,  

{\rm(2)} $I_{\mu\mid N}(X_N(\la);q,t)=0$ for all $\la\ne\mu$ with $|\la|\le|\mu|$, 

{\rm(3)}  $I_{\mu\mid N}(X_N(\mu);q,t)\ne0$.
\end{proposition}

\begin{proof} 
The set $\{X_N(\la): \la\in\Y(N)\}\subset\FF^N$ is a nondegenerate grid  in the sense of Definition 2.1 in Okounkov \cite{Ok-AAM}. Then the desired claim follows from \cite[Proposition 2.3]{Ok-AAM}. (Note that the proof of that proposition follows the argument due to Sahi \cite{Sahi1}.) 
\end{proof}

\begin{proposition}\label{prop2.D}
{\rm(i)} The leading homogeneous component of  $I_{\mu\mid N}(\ccdot;q,t)$ is proportional to $P_{\mu\mid N}(\ccdot;q,t)$.  

{\rm(ii)} If a diagram $\la\in\Y(N)$ does not contain $\mu$, then $I_{\mu\mid N}(X_N(\la);q,t)=0$. 
\end{proposition}

\begin{proof}
Claim (i) was proved in \cite{Knop}, \cite{Sahi2}, and \cite{Ok-CM}. Claim (ii) was proved in \cite{Knop} and \cite{Ok-CM}; it is called the \emph{extra vanishing property}.
\end{proof}

In the sequel we specify the normalization of  $I_{\mu\mid N}(\ccdot;q,t)$ by requiring that 
\begin{equation}\label{eq2.A}
I_{\mu\mid N}(x_1,\dots,x_N;q,t)=P_{\mu\mid N}(x_1,\dots,x_N;q,t)+\textrm{\rm lower degree terms.}
\end{equation}

Because the choice of interpolation nodes varies from paper to paper, let us indicate the correspondence of notations. Our polynomial $I_{\mu\mid N}(x_1,\dots,x_N;q,t)$ is equal to:

\smallskip

$\bullet$ $P_\mu(x_1,\dots,x_N)$ with parameters $(1/q,1/t)$, in the notation of Knop \cite{Knop};

\smallskip

$\bullet$ $P^*_\mu(x_1,x_2t^{-1},\dots,x_N t^{1-N}; 1/q,1/t)$, in the notation of Okounkov \cite{Ok-MRL}, \cite{Ok-CM};

\smallskip

$\bullet$ $t^{(N-1)|\mu|}R_\mu(x_1 t^{1-N}, \dots, x_N t^{1-N}; q,t)$, in the notation of Sahi \cite{Sahi2}.

\smallskip

It follows from \eqref{eq2.A} that the interpolation polynomials form a basis of the algebra $\Sym_\FF(N)$. This basis is inhomogeneous but consistent  with the natural filtration by degree. 

Recall the \emph{combinatorial formula} for the Macdonald polynomials:
\begin{equation}\label{eq2.D}
P_{\mu\mid N}(x_1,\dots,x_N;q,t)=\sum_{T\in\Tab(\mu,N)}\psi_T(q,t) \prod_{(i,j)\in\mu}x_{T(i,j)},
\end{equation}
where $\Tab(\mu,N)$ is the set of semistandard tableaux of shape $\mu$ and taking values in $\{1,\dots,N\}$, and $\psi_T(q,t)$ are some rational functions in $(q,t)$, see \cite[ch. VI, \S7 ]{M}.

The next proposition is an analogue of \eqref{eq2.D} for the interpolation polynomials. It is convenient to replace $\Tab(\mu,N)$ by the set $\RTab(\mu,N)$ of \emph{reverse} semistandard tableaux of shape $\mu$ with values in $\{1,\dots,N\}$: the difference from $\Tab(\mu,N)$ is that the conventional ordering of $\{1,\dots,N\}$ is reversed, cf. \cite{OO-AA}. That is, a tableau $T\in\RTab(\mu,N)$ is a filling of the boxes $(i,j)\in\mu$ with numbers $T(i,j)\in\{1,\dots,N\}$ in such a way that these numbers weakly decay along the rows and strictly decay down the columns. 

\begin{proposition}\label{prop2.H}
We have 
\begin{equation}\label{eq2.E}
I_{\mu\mid N}(x_1,\dots,x_N;q,t)=\sum_{T\in\RTab(\mu,N)} \psi_T(q,t)\prod_{(i,j)\in\mu}(x_{T(i,j)}-q^{1-j}t^{T(i,j)+i-2}).
\end{equation}
\end{proposition}

\begin{proof}
Okounkov's combinatorial formula for the shifted symmetric polynomials $P^*_\mu(x_1,\dots,x_N;q,t)$ has the form \cite[(1.4)]{Ok-CM}
$$
P^*_\mu(x_1,\dots,x_N;q,t)=\sum_{T\in\RTab(\mu,N)} \psi_T(q,t)\prod_{(i,j)\in\mu}t^{1-T(i,j)}(x_{T(i,j)}-q^{j-1}t^{1-i}).
$$
From this formula we obtain
\begin{gather*}
I_{\mu\mid N}(x_1,\dots,x_N;q,t)=P^*_\mu(x_1,x_2t^{-1},\dots,x_N t^{1-N}; 1/q,1/t)\\
=\sum_{T\in\RTab(\mu,N)} \psi_T(q^{-1},t^{-1})\prod_{(i,j)\in\mu}(x_{T(i,j)}-q^{1-j}t^{T(i,j)+i-2}).
\end{gather*}
Next, the definition of $\psi_T(q,t)$ (see \cite[ch. VI, (6.24) (ii) and (7.11${}'$)]{M}) implies that
$$
\psi_T(q^{-1},t^{-1})=\psi_T(q,t).
$$
This yields the desired result.
\end{proof}

\begin{remark}
As usual, the combinatorial formula \eqref{eq2.E} is equivalent to a \emph{branching rule}; in our situation it can be written as 
\begin{multline}\label{eq2.E1}
I_{\mu\mid N}(x_1,\dots,x_N;q,t)\\
=\sum_{\nu\prec\mu}\psi_{\mu/\nu}(q,t)t^{|\nu|}\prod_{(i,j)\in\mu/\nu}(x_1-q^{1-j}t^{i-1})\cdot I_{\nu\mid N-1}(x_2t^{-1},\dots,x_N t^{-1};q,t),
\end{multline}
where $\nu\prec\mu$ means that $\nu\subseteq\mu$ and $\mu/\nu$ is a horizontal strip. 
\end{remark}

Note that in the special case $t=q$ the interpolation polynomials can be given by an explicit determinantal formula, see Section \ref{sect4}  below. 

The following proposition describes the expansion of Macdonald polynomials in the basis of interpolation polynomials.  

\begin{proposition}\label{prop2.E}
We have
$$
\frac{P_{\mu\mid N}(\ccdot;q,t)}{P_{\mu\mid N}(1,t,\dots,t^{N-1};q,t)}=\sum_{\nu\subseteq\mu} \frac{I_{\nu\mid N}(X_N(\mu);q,t)}{I_{\nu\mid N}(X_N(\nu);q,t)} \frac{I_{\nu\mid N}(\ccdot;q,t)}{P_{\nu\mid N}(1,t,\dots,t^{N-1};q,t)}.
$$
\end{proposition}

\begin{proof}
This is a simple reformulation of (a particular case of) Okounkov's \emph{binomial formula}. See formula \cite[(1.11)]{Ok-MRL}, which in turn is a particular case of a more general result \cite[p. 536]{Ok-MRL}. Another approach is given in \cite{Ok-FAA}.
\end{proof}

There exists also another version of binomial formula (see \cite[(1.12)]{Ok-MRL}), which  describes the inverse expansion of interpolation polynomials in Macdonald polynomials. 

\begin{proposition}\label{prop2.F}
Let $N\ge2$ and $\mu\in\Y(N)$, then  
\begin{equation}\label{eq2.C}
I_{\mu\mid N}(x_1,\dots,x_N;q,t)\big |_{x_N=t^{N-1}}=\begin{cases}  I_{\mu\mid N-1}(x_1,\dots,x_{N-1};q,t), & \mu\in\Y(N-1);\\ 0, & \text{\rm otherwise}.\end{cases}
\end{equation}
\end{proposition}

\begin{proof}
Let $L=L(x_1,\dots,x_{N-1})$ stand for the polynomial on the left. Clearly, it is symmetric. Let $\la\in\Y(N-1)$ be arbitrary. 

Suppose first that  $\mu\in\Y(N-1)$, so that $\mu_N=0$. Observe that the last coordinate of $X_N(\la)$ is equal to $t^{N-1}$. Using this it is readily checked that $L$ satisfies all the properties of the polynomial $I_{\mu\mid N-1}(\ccdot;q,t)$ listed in Proposition \ref{prop2.C} and hence must be proportional to it. Next, in the expansion of both $I_{\mu\mid N}(\ccdot;q,t)$ and $P_{\mu\mid N}(\ccdot;q,t)$ on the monomials, the leading (in the lexicographic order) term is the monomial $x_1^{\mu_1}\dots x_{N-1}^{\mu_{N-1}}$, which enters with coefficient $1$. This term is not affected by the specialization $x_N=t^{N-1}$ and remains the leading term in $L$. This implies that $L=I_{\mu\mid N-1}(\ccdot;q,t)$. 

Suppose now that  $\mu_N>0$. Then $\mu$ is not contained in $\la$, hence (by virtue of the extra vanishing property) $I_{\mu\mid N}(X_N(\la))=0$. This implies $L(X_{N-1}(\la))=0$, and since $\la\in\Y(N-1)$ may be arbitrary we conclude that $L=0$. 
\end{proof}

The property \eqref{eq2.C} of the interpolation polynomials will be referred to as \emph{quasi-stability}. 

\subsection{Interpolation symmetric functions}
The quasi-stability property of interpolation polynomials makes it possible to lift them to the algebra $\Sym_\FF$ (cf. \cite{Ols-2017}). 

Recall that the algebra of symmetric functions is isomorphic to the algebra of polynomials in the power sums $p_1,p_2,\dots$. We define, for  each $N=1,2,\dots$, an algebra morphism $\phi_N:\Sym_\FF\to\Sym_\FF(N)$ by
$$
\phi_N(p_n)=x_1^n+\dots+x_N^n+\frac{t^{Nn}}{1-t^n}, \qquad n=1,2,\dots\,.
$$
This definition is prompted by the fact that 
$$
\frac{t^{Nn}}{1-t^n}=\sum_{i>N}(t^{i-1})^n,
$$
provided that $t$ is thought of as a complex parameter with $|t|<1$. Thus, informally, $\phi_N$ might be be viewed as the specialization at $x_i=t^{i-1}$ for $i>N$. 

\begin{proposition}\label{prop.def}
For every partition $\mu$ there exists a unique symmetric function $I_\mu(\ccdot;q,t)\in\Sym_\FF$ such that 
\begin{equation}\label{eq.Imu}
\phi_N(I_\mu(\ccdot;q,t))=I_{\mu\mid N}(\ccdot;q,t), \qquad N\ge\ell(\mu).
\end{equation}
\end{proposition}

\begin{proof} 
We begin with a general construction. Fix an arbitrary infinite sequence $a_1,a_2,\dots$ of elements of the base field $\FF$. For $N=1,2,\dots$ we define a homomorphism $\pi_N: \Sym_\FF(N)\to\Sym_\FF(N-1)$ as the specialization $x_N=a_N$ (we assume $\Sym_\FF(0)=\FF$).  Let $\Sym_\FF':=\varprojlim(\Sym_\FF(N),\pi_N)$ be the projective limit taken in the category of filtered algebras, with respect to the projections $\pi_N$. Observe that the associated graded algebra $\operatorname{gr}(\Sym_\FF')$ is canonically isomorphic to $\Sym_\FF$.

Suppose now that for any $n=1,2,\dots$, the infinite formal series 
$a_1^n+a_2^n+\dots$ 
can be regularized in the sense that one can exhibit elements $b_{N,n}\in\FF$ such that
\begin{equation}\label{eq.Imu1}
b_{N,n}-b_{N+1,n}=a_N^n, \qquad \text{for all $N,\, n=1,2,\dots$}\,.
\end{equation}
Then we define, for  each $N=1,2,\dots$, an algebra morphism $\psi_N:\Sym_\FF\to\Sym_\FF(N)$ by
$$
\psi_N(p_n)=x_1^n+\dots+x_N^n+b_{N+1,n}, \qquad n=1,2,\dots\,.
$$

Observe that
\begin{equation}\label{eq.Imu2}
\pi_N\circ\psi_N=\psi_{N-1}, \qquad N=1,2,\dots\,.
\end{equation}
Indeed, it suffices to check the equality $\pi_N(\psi_N(p_n))=\pi_{N-1}(p_n)$, but it reduces to \eqref{eq.Imu1}.

From \eqref{eq.Imu2} it follows that the projections $\psi_N$ give rise to an algebra morphism $\psi: \Sym_\FF\to\Sym_\FF'$, which preserves filtration. Passing to the associated graded algebra we see that $\psi$ is an isomorphism. 

It follows that whenever we have a sequence  of elements $\{I_N\in\Sym_\FF(N)\}$ of  bounded degree and such that  
\begin{equation}\label{eq.Imu3}
I_N(x_1,\dots,x_{N-1}, a_N)=I_{N-1}(x_1,\dots,x_{N-1}) \quad \text{for all $N$ large enough},
\end{equation}
then there exists a unique element $I\in\Sym_\FF$ such that
$$
\psi_N(I)=I_N \quad \text{for all $N$ large enough}. 
$$

Now we apply this general argument for the concrete sequence $a_N:=t^{N-1}$. Set
$$
b_{N,n}:=\frac{t^{(N-1)n}}{1-t^n}.
$$ 
Then
$$
b_{N,n}-b_{N+1,n}=\frac{t^{(N-1)n}}{1-t^n} - \frac{t^{Nn}}{1-t^n}=t^{(N-1)n}=a_N^n,
$$
so that the relation \eqref{eq.Imu1} is satisfied. Take $I_N=I_{\mu\mid N}(\ccdot;q,t)$. The relation \eqref{eq.Imu3} holds by virtue of the quasi-stability property. Then we take $\psi_N:=\phi_N$ and obtain the desired result. 
\end{proof}

We call the elements $I_\mu(\ccdot;q,t)\in\Sym_\FF$ the \emph{interpolation Macdonald symmetric functions}.

\begin{remark}
All the statements of Propositions \ref{prop2.C}, \ref{prop2.D}, and \ref{prop2.E} can be extended to these functions, with evident modifications. Namely, partitions $\mu$ and $\la$ may range over the whole set $\Y$, the interpolation nodes are defined by
$$
X(\la):=(q^{-\la_1}, q^{-\la_2}t, q^{-\la_3}t^2,\dots), \quad \la\in\Y,
$$
the Macdonald polynomials are replaced by the Macdonald symmetric functions, and the binomial formula takes the form
$$
\frac{P_{\mu}(\ccdot;q,t)}{P_{\mu}(1,t,t^2,\dots;q,t)}=\sum_{\nu\subseteq\mu} \frac{I_\nu(X(\mu);q,t)}{I_{\nu}(X(\nu);q,t)} \frac{I_{\nu}(\ccdot;q,t)}{P_\nu(1,t,t^2,\dots;q,t)}.
$$
Note that for any $F\in\Sym_\FF$, the quantity $F(X(\la))\in\FF$ is well defined, with the understanding that
$$
F(x_1,\dots,x_N, t^N,t^{N+1},\dots;q,t):=\phi_N(F)(x_1,\dots,x_N).
$$
In particular, $F(1,t,t^2,\dots)\in\FF$ is understood as the result of the specialization $\Sym_\FF\to\FF$ sending $p_n$ to $(1-t^n)^{-1}$ for $n=1,2,\dots$\,.
\end{remark}

The following result is a reformulation of \cite[(2.10)]{Ok-MRL}. Its proof is similar to \cite[second proof of Theorem 12.1]{OO-AA}.

\begin{proposition}[Generating function for one-row interpolation functions]\label{prop2.B}
We have
\begin{equation}
1+\sum_{m=1}^\infty \frac{(t;q)_m}{(q;q)_m}\frac{I_{(m)}(x_1,x_2,\dots;q,t)}{(u-q^{-1})\dots(u-q^{-m})}=\prod_{i=1}^\infty\frac{(x_iu^{-1}t;q)_\infty}{(x_iu^{-1};q)_\infty}\frac{(u^{-1}t^{i-1};q)_\infty}{(u^{-1}t^i;q)_\infty},
\end{equation}
where $u$ is a formal variable and both sides are regarded as elements of\/ $\Sym_\FF[[u^{-1}]]$.
\end{proposition}

\begin{proof}
First of all, observe that both sides are well defined as elements of the algebra $\Sym_\FF[[u^{-1}]]$ (for the right-hand side we use Remark \ref{Pochhammer}). 

Proceeding to the proof, abbreviate $X:=(x_1,x_2,\dots)$ and denote the right-hand side by $R(X,u)$. It is an element of $\Sym_\FF[[u^{-1}]]$ with constant term $1$. It follows that $R(X,u)$ can be expanded, in a unique way,  into a series of the form  
$$
1+\sum_{m=1}^\infty \frac{f_m(X)}{(u-q^{-1})\dots(u-q^{-m})},
$$
where $f_m\in\Sym_\FF$. We have to show that 
$$
f_m(X)=\frac{(t;q)_m}{(q;q)_m}I_{(m)}(X;q,t), \qquad m=1,2,\dots\,.
$$
To do this, it suffices to prove two claims:

1.  $\deg f_m\le m$ and (denoting by $[f_m(X)]$ its degree $m$ homogeneous component) one has 
$$
[f_m(X)]=\frac{(t;q)_m}{(q;q)_m}P_{(m)}(X;q,t).
$$

2.  $f_m(X(\la))=0$ if $\la_1<m$.

To check the first claim we multiply all $x_i$'s and $u$ by a scalar $r$ and let $r$ go to infinity; then we obtain
$$
1+\sum_{m=1}^\infty \frac{[f_m(X)]}{u^m}=\prod_{i=1}^\infty\frac{(x_iu^{-1}t;q)_\infty}{(x_iu^{-1};q)_\infty}.
$$
Comparing this with formulas in Macdonald's book (see \cite[Chapter VI, (2.8), (5.5), (4.12), and (6.20)]{M} we see that
$$
[f_m(X)]=Q_{(m)}(X;q,t)=b_{(m)}(q,t) P_{(m)}(X;q,t)=\frac{(t;q)_m}{(q;q)_m}P_{(m)}(X;q,t),
$$
as desired.

The second claim can be rephrased as follows: let $\la\in\Y$ by arbitrary; then the expansion of $R(X(\la),u)$  terminates at most at $m=\la_1$. 

Let us prove this. We are going to show that 
$$
R(X(\la),u)=\frac{F(u)}{(u-q^{-1})\dots(u-q^{-\la_1})}
$$
where $F(u)$ is a monic polynomial of degree $\la_1$. Any such fraction can be written in the form 
$$
1+\sum_{m=1}^{\la_1}\frac{c_m}{(u-q^{-1})\dots(u-q^{-m})},
$$
which is just we need. 

Now we compute $R(X(\la),u)$:  
\begin{gather*}
R(X(\la),u)=\prod_{i=1}^N\frac{(q^{-\la_i}t^iu^{-1};q)_\infty}{(t^iu^{-1};q)_\infty}\frac{(t^{i-1}u^{-1})_\infty}{(q^{-\la_i}t^{i-1}u^{-1};q)_\infty}\\
=\frac{\prod_{i=1}^N(u-q^{-1}t^i)\dots(u-q^{-\la_i}t^i)}{\prod_{i=1}^N(u-q^{-1}t^{i-1})\dots(u-q^{-\la_i}t^{i-1})}=:\frac{\prod_{i=1}^NA_i(u)}{\prod_{i=1}^N B_i(u)}
\end{gather*}

This expression is the ratio of two monic polynomials of the same degree equal to $|\la|$. The polynomial $B_1(u)$ is exactly the denominator that we want.  Next, for each $i=2,\dots,N$, we have $\la_i\le \la_{i-1}$, hence all the factors in $B_i(u)$ are contained among the factors of $A_{i-1}(u)$, so that they are canceled out. It follows that $R(X(\la),u)$ has the desired form.

This completes the proof. 
\end{proof}

\section{Dual interpolation functions}\label{sect3}

\subsection{Biorthogonal systems}
We begin with a little formalism. 
Let $V=\bigoplus_{n=0}^\infty V_n$ be a graded vector space over a field, and suppose that all components $V_n$ are nonzero and have finite dimension. The grading of $V$ gives rise to two natural filtrations: one is ascending, $\{V_{\le k}\}$, and the other is descending, $\{V_{\ge k}\}$:
$$
V_{\le k}:=\bigoplus_{n=0}^k V_n, \qquad V_{\ge k}:=\bigoplus_{n=k}^\infty V_n.
$$

Denote by $\wh V$ the completion of $V$ with respect to the descending filtration. That is, elements of $\wh V$ are arbitrary formal infinite series $\sum_{n=0}^\infty f_n$, where $f_n\in V_n$. Let $\wh V_{\ge k}$ denote the completion of $V_{\ge k}$, that is,
$$
\wh V_{\ge k}=\left\{\sum_{n}f_n: f_0=\dots=f_{k-1}=0\right\}\subset \wh V.
$$

The space $\wh V$ is endowed with the topology in which the subspaces $\wh V_{\ge k}$ form a fundamental system of neighborhoods of the zero. With this topology, $\wh V$  is a separable topological vector space. 

Next, assume that $V$ is endowed with a nondegenerate bilinear symmetric scalar product $\langle\ccdot,\ccdot\rangle$ such that the components $V_n$ are pairwise orthogonal. Such a scalar product extends to a pairing between $V$ and $\wh V$, which allows one to identify $\wh V$ with the dual space to $V$. 

Let $\{v_\mu\}$ be an arbitrary basis in $V$, not necessarily homogeneous, but consistent with the ascending filtration (here $\mu$ runs through a countable set of indices). For any set $\{a_\mu\}$ of nonzero scalars there exists a unique family $\{\wh v_\mu\in\wh V\}$ such that $\langle v_\mu,\wh v_\nu\rangle=0$ for any pair $\mu,\nu$ of distinct indices and $\langle v_\mu,\wh v_\mu\rangle=a_\mu$. We say that $\{v_\mu\}$ and $\{\wh v_\mu\}$ form a \emph{biorthogonal system}.

For an arbitrary collection $\{c_\mu\}$ of scalars, the infinite series $\sum_\mu c_\mu \wh v_\mu$ converges in the topology of the space $\wh V$, and every element of $\wh V$ is represented by such a series, in a unique way.  In this sense, $\{\wh v_\mu\}$ is a topological basis of $\wh V$. 

Finally, note that if $V$ is a graded algebra, then $\wh V$ is also an algebra.

\subsection{The algebra $\Symd$ and the dual functions $\Id_\mu(\ccdot;q,t)$}

We apply the above formalism to $V:=\Sym_\FF$ and denote the corresponding completion $\wh V$ by $\Symd$. The algebra $\Symd$ is isomorphic to the $\FF$-algebra of symmetric formal power series in countably many variables, with no restriction on the degree of monomials (whereas the elements of the subalgebra $\Sym_\FF\subset\Symd$ must have bounded degree). The formal variables will be denoted as $y_1,y_2,\dots$\,. 
  
We equip $\Sym_\FF$ with the Macdonald scalar product  $\langle\ccdot,\ccdot\rangle_{q,t}$, see \cite[ch. VI, \S2]{M}. It determines a pairing between $\Sym_\FF$ and $\Symd$. Thus, we may view $\Symd$ as the algebraic dual space to $\Sym_\FF$.

\begin{definition}\label{def3.A}
The elements $\Id_\mu(\ccdot;q,t)\in\Symd$, where $\mu\in\Y$,  are defined as the dual family with respect to the basis $\{I_\mu(\ccdot;q,t)\}$. That is, 
$$
\langle I_\mu(\ccdot;q,t), \Id_\nu(\ccdot;q,t)\rangle_{q,t}=\de_{\mu\nu}.
$$
We call them the \emph{dual interpolation functions}. 
\end{definition}

By the very definition, 
$$
\Id_\mu(y_1,y_2,\dots;q,t)=Q_\mu(y_1,y_2,\dots;q,t)+\textrm{higher degree terms in $y_1,y_2,\dots$\,.}
$$
According to the general formalism, the dual functions form a topological basis in $\Symd$.

\subsection{The Cauchy identity}

\begin{proposition}\label{prop3.A}
The interpolation functions and their duals are related by the following analogue of the Cauchy identity{\rm:}
\begin{equation}\label{eq3.C}
\sum_{\mu\in\Y}I_\mu(x_1,x_2,\dots;q,t)\Id_\mu(y_1,y_2,\dots;q,t)=\prod_{i=1}^\infty\prod_{j=1}^\infty\frac{(x_iy_jt;q)_\infty}{(x_iy_j;q)_\infty},
\end{equation}
where both sides are regarded as formal power series in the indeterminates $x_i$ and $y_j$, with coefficients in the field\/ $\FF$.
\end{proposition}

\begin{proof} 
By the very definition of the dual functions,
$$
\sum_{\mu\in\Y}I_\mu(x_1,x_2,\dots;q,t)\Id_\mu(y_1,y_2,\dots;q,t)=\sum_{\mu\in\Y}P_\mu(x_1,x_2,\dots;q,t) Q_\mu(y_1,y_2,\dots;q,t).
$$
By the definition of the scalar product $\langle\ccdot,\ccdot\rangle_{q,t}$, the right-hand side is equal to the double product in \eqref{eq3.C}, see \cite[ch. VI, \S2]{M}. 
\end{proof}

\subsection{Restriction to $N$ variables}

Now we apply the same formalism to the algebra $\Sym_\FF(N)$ (where $N=1,2,\dots)$ and denote by  $\Symd(N)$ its completion. The elements of $\Symd(N)$ are arbitrary symmetric formal power series in $N$ variables $y_1,\dots,y_N$.  There is a natural homomorphism $\Symd\to\Symd(N)$, it is defined as the specialization $y_{N+1}=y_{N+2}=\dots=0$. 

Let $N$ be fixed. For every $\mu\in\Y$, let $\Id_{\mu\mid N}(\ccdot;q,t)\in\Symd(N)$ be defined as the image of $\Id_\mu(\ccdot;q,t)\in\Symd$ under this homomorphism. 

Obviously,
$$
\Id_{\mu\mid N}(y_1,\dots,y_N;q,t)=Q_{\mu\mid N}(y_1,\dots,y_N;q,t)+\textrm{higher degree terms}, \qquad \forall \mu\in\Y(N).
$$ 
It follows that the elements $\Id_{\mu\mid N}(y_1,\dots,y_N;q,t)$ with $\ell(\mu)\le N$ form a topological basis in $\Symd(N)$.

\begin{lemma}
If $\ell(\mu)>N$, then $\Id_{\mu\mid N}(\ccdot;q,t)=0$.
\end{lemma}

\begin{proof}
Set $X=(x_1,x_2,\dots)$. The Cauchy identity \eqref{eq3.C} gives
$$
\sum_{\mu\in\Y}I_\mu(X;q,t)\Id_{\mu\mid N}(y_1,\dots,y_N;q,t)=\prod_{i=1}^\infty\prod_{j=1}^N\frac{(x_i y_jt;q)_\infty}{(x_i y_j;q)_\infty}.
$$
Hence $\Id_{\mu\mid N}(y_1,\dots,y_N;q,t)$ coincides with the coefficient of $I_\mu(X;q,t)$ in the expansion of the right-hand side in the interpolation symmetric functions.

On the other hand, the right-hand side equals
$$
\sum_{\la\in\Y(N)}P_\la(X;q,t) Q_{\la\mid N}(y_1,\dots,y_N;q,t).
$$
Next, for each $\la\in\Y(N)$, the expansion of $P_\la(X;q,t)$ in the interpolation functions involves  only elements $I_\mu(X;q,t)$ with indices $\mu\in\Y(N)$.

This completes the proof. 
\end{proof}

\begin{lemma} The following truncated version of Cauchy identity holds.
\begin{multline*}
\sum_{\mu: \, \ell(\mu)\le\min(N,K)}I_{\mu\mid N}(x_1,\dots,x_N;q,t)\Id_{\mu\mid K}(y_1,\dots,y_K;q,t)\\
=\prod_{i=1}^N\prod_{j=1}^K\frac{(x_i y_jt;q)_\infty}{(x_i y_j;q)_\infty}\cdot\prod_{j=1}^K\frac1{(y_jt^N;q)_\infty}.
\end{multline*}
\end{lemma}

\begin{proof}
Consider the Cauchy identity \eqref{eq3.C} and specialize $x_i=t^{i-1}$ for $i>N$ and $y_j=0$ for $j>K$.   On the left we obtain the desired sum over $\mu\in\Y(N)\cap\Y(K)$. On the right we obtain
$$
\prod_{i=1}^N\prod_{j=1}^K\frac{(x_i y_jt;q)_\infty}{(x_i y_j;q)_\infty}\cdot\prod_{i=N+1}^\infty\prod_{j=1}^K\frac{(t^iy_j;q)_\infty}{(t^{i-1}y_j;q)_\infty}.
$$
Next,
$$
\prod_{i=N+1}^\infty\prod_{j=1}^K\frac{(t^iy_j;q)_\infty}{(t^{i-1}y_j;q)_\infty}=\prod_{j=1}^K\prod_{i=N+1}^\infty\frac{(t^iy_j;q)_\infty}{(t^{i-1}y_j;q)_\infty}=\prod_{j=1}^K \frac1{(y_j t^N ;q)_\infty},
$$
which completes the proof.
\end{proof}

\subsection{Modified dual functions}

\begin{definition}\label{def3.B}
We slightly change Definition \ref{def3.A} and introduce the \emph{modified} dual interpolation functions by setting
$$
\Idm_\mu(u_1,u_2,\dots; q,t):=\Id_\mu(u_1^{-1},u_2^{-1},\dots;q,t)\cdot \prod_{j=1}^\infty(u_j^{-1};q)_\infty\in\FF[[u_1^{-1}, u_2^{-1},\dots]], 
$$
and
$$
\Idm_{\mu\mid N}(u_1,\dots,u_N; q,t):=\Id_{\mu\mid N}(u_1^{-1},\dots,u_N^{-1};q,t) \cdot \prod_{j=1}^N(u_j^{-1};q)_\infty\in\FF[[u_1^{-1},\dots,u_N^{-1}]].
$$
\end{definition}

The advantage of this modification is that the $\Idm_{\mu\mid N}(u_1,\dots,u_N; q,t)$ are \emph{rational} functions, see Corollary \ref{cor5.B} below.

The Cauchy identity \eqref{eq3.C} is modified as follows:
\begin{equation}\label{eq3.A}
\sum_{\mu\in\Y}I_\mu(x_1,x_2,\dots;q,t)\Idm_\mu(u_1,u_2,\dots;q,t)=\prod_{i,j=1}^\infty\frac{(x_iu^{-1}_jt;q)_\infty}{(x_iu^{-1}_j;q)_\infty}\cdot\prod_{j=1}^\infty(u^{-1}_j;q)_\infty,
\end{equation}
and its truncated version takes the form
\begin{multline}\label{eq3.B}
\sum_{\mu:\, \ell(\mu)\le\min(N,K)}I_{\mu\mid N}(x_1,\dots,x_N;q,t)\Idm_{\mu\mid K}(u_1,\dots,u_K;q,t)\\
=\prod_{i=1}^N\prod_{j=1}^K\frac{(x_iu^{-1}_jt;q)_\infty}{(x_iu^{-1}_j;q)_\infty}\cdot\prod_{j=1}^K\frac{(u^{-1}_j;q)_\infty}{(u^{-1}_jt^N;q)_\infty}.
\end{multline}

\section{The case $t=q$}\label{sect4}

In the special case $t=q$ the interpolation polynomials and the dual functions in $N$ variables admit explicit  determinantal expressions. Here we present them. It is convenient to introduce first a larger family of functions.  

\begin{definition}[Schur-type functions]\label{def2.A}
Let $f_0(z)\equiv1, f_1(z), f_2(z),\dots$ be an infinite sequence of
functions of a variable $z$. For a partition $\mu\in\Y$ and $N=1,2,\dots$ we define
a symmetric function of $N$ variables $z_1,\dots,z_N$ by
\begin{equation*}
f_{\mu|N}(z_1,\dots,z_N):=\begin{cases}\dfrac{\det[f_{\mu_i+N-i}(z_j)]_{i,j=1}^N}
{\det[f_{N-i}(z_j)]_{i,j=1}^N}, & \ell(\mu)\le N,\\
0, &\text{otherwise}.
\end{cases}
\end{equation*}
\end{definition}

If $f_m$ is a monic polynomial of degree $m$ for each $m$, then the denominator  is equal to the Vandermonde determinant
$$
V(z_1,\dots,z_N):=\prod_{1\le i<j\le N}(z_i-z_j),
$$
which implies that the corresponding $N$-variate functions are symmetric polynomials. In particular, if $f_m(z)=z^m$, then we obtain the ordinary Schur polynomials $s_{\mu\mid N}(z_1,\dots,z_N)$.

\begin{example}[Multiparameter Schur polynomials, cf. Macdonald \cite{M-SLC}, Molev \cite{Molev}]
Let $(c_0,c_1,c_2,\dots)$ be an infinite sequence of parameters and
\begin{equation*}
f_m(x)=(x\mid c_0,c_1,\dots)^m:=(x-c_0)\dots(x-c_{m-1}).
\end{equation*}
 The corresponding Schur-type functions  are denoted by $s_{\mu\mid N}(x_1,\dots,x_N\mid c_0,c_1,\dots)$ and we call them 
 the \emph{multiparameter Schur polynomials}: 
\begin{equation*}
s_{\mu\mid N}(x_1,\dots,x_N\mid c_0,c_1,\dots)
:=\frac{\det[(x_j\mid c_0,c_1,\dots)^{\mu_i+N-i}]}{V(x_1,\dots,x_N)},
\end{equation*}
If $c_0=c_1=\dots=0$, they turn into the ordinary Schur polynomials, and in the case $c_i=i$ they give the \emph{factorial Schur polynomials} introduced by Biedenharn and Louck \cite{BL}, \cite{BL-PNAS} and further studied in a number of works, see, e.g.,  Chen and Louck \cite{CL}, Goulden and Hamel \cite{GH}, Macdonald \cite{M-SLC}. (Note that in some recent works, the polynomials $s_{\mu\mid N}(x_1,\dots,x_N\mid c_0,c_1,\dots)$ are also called ``factorial Schur'', see, e.g., Bump, McNamara, and Nakasuji \cite{Bump}.)
\end{example}

The interpolation polynomials with $t=q$ are a special case of multiparameter Schur polynomials corresponding to the sequence $c_n:=q^{N-n-1}$, $n=0,1,2,\dots$:

\begin{proposition}
One has
\begin{multline}\label{eq4.H}
I_{\mu\mid N}(x_1,\dots,x_N;q,q)=s_{\mu\mid N}(x_1,\dots,x_N\mid q^{N-1}, q^{N-2},\dots)\\
=\frac{\det\left[(x_j-q^{N-1})(x_j-q^{N-2})\dots(x_j-q^{-\mu_i+i})\right]_{i,j=1}^N}{V(x_1,\dots,x_N)}.
\end{multline}
\end{proposition}

\begin{proof}
Let $R=R(x_1,\dots,x_N)$ denote the polynomial on the right. It suffices to check that $R$ possesses the same extra vanishing property as $I_{\mu\mid N}(\ccdot; q,q)$ and that the leading homogeneous component of $R$  coincides with the Schur polynomial $s_{\mu\mid N}(\ccdot)$. The latter claim is evident, let us prove the former one. To do this we use the same simple argument as in \cite[\S2.4]{Ok-TG} and \cite[Theorem 3.1]{OO-AA}.

Let $\la\in\Y(N)$. In the case $t=q$ we have 
$$
X_N(\la)=(q^{-\la_1}, q^{-\la_2+1},\dots, q^{-\la_N+N-1}).
$$
Suppose that $\la$ does not contain $\mu$, which means that there exists an index $k$ such that $\la_k<\mu_k$. Set $(x_1,\dots,x_N)=X_N(\la)$, so that $x_j=q^{-\la_j+j-1}$ for $j=1,\dots,N$. Then the $(i,j)$th entry of the determinant in \eqref{eq4.H} vanishes for all pairs $(i,j)$ such that $i\le k\le j$. This in turn implies that the determinant itself equals $0$. 
\end{proof}

Note that in the notation of Molev \cite{Molev},
$$
I_{\mu\mid N}(x_1,\dots,x_N;q,q)=s_\mu(x_1\dots,x_N\Vert a),
$$
where
$$
a=(a_n)_{n\in\Z}, \qquad a_n:=q^{n-1}.
$$

\begin{example}[Dual multiparameter  Schur functions]
Set
\begin{equation*}
f_m(u)=\frac1{(u\mid c_1,c_2,\dots)^m}
\end{equation*}
(notice a shift in the indexation of the parameters). The corresponding $N$-variate functions are denoted by $\si_{\mu\mid N}(u_1,\dots,u_N\mid c_1,c_2,\dots)$:
\begin{equation}\label{eq4.sigma}
\si_{\mu\mid N}(u_1,\dots,u_N\mid c_1,c_2,\dots)
:=\frac{\det\left[\dfrac1{(u_j\mid c_1,c_2,\dots)^{\mu_i+N-i}}\right]_{i,j=1}^N}
{\det\left[\dfrac1{(u_j\mid c_1,c_2,\dots)^{N-i}}\right]_{i,j=1}^N}.
\end{equation}
We call them the ($N$-variate) \emph{dual Schur functions} or \emph{$\si$-functions}, for short (cf. Molev \cite{Molev}). If $c_1=c_2=\dots=0$, they turn into the conventional Schur polynomials in variables $u_1^{-1},\dots,u_N^{-1}$. 
\end{example}

\begin{lemma}\label{lemma4.A}
The $\si$-functions possess the following stability property{\rm:}
\begin{equation*}
\si_{\mu\mid N}(u_1,\dots,u_N\mid c_1,c_2,\dots)\big | _{u_N=\infty}
=\de_{\mu_N,0}\,\si_{\mu\mid N-1}(u_1,\dots,u_{N-1}\mid c_2,c_3,\dots).
\end{equation*}
\end{lemma}

\begin{proof}
We use two obvious relations,
$$
\frac1{(u\mid c_1,c_2,\dots)^m}\Bigg|_{u=\infty}=\begin{cases} 0, & m=1,2,\dots, \\ 1, & m=0,\end{cases}
$$
and
$$
(u\mid c_1,c_2,\dots)^m=(u-c_1)(u\mid c_2,c_3,\dots)^{m-1}, \qquad m=1,2,\dots\,.
$$

Consider the matrix in the numerator of \eqref{eq4.sigma}, where we specialize $u_N=\infty$. The first relation shows that $N$th row has the form $(0,\dots,0)$ or $(0,\dots,0,1)$ depending on whether $\mu_N>0$ or $\mu_N=0$. In the former case the determinant vanishes, and in the latter case it reduces to a determinant of size $(N-1)\times(N-1)$. The determinant in the denominator also reduces to a determinant of size $(N-1)\times(N-1)$. Both determinants produce the product $(u_1-c_1)\dots(u_{N-1}-c_1)$, which is canceled. This finally gives the desired result.
\end{proof}

\begin{lemma}\label{lemma4.B}
The $\si$-functions can also be written in the form
\begin{multline}\label{eq.A}
\si_{\mu\mid N}(u_1,\dots,u_N\mid c_1,c_2,\dots)
=(-1)^{N(N-1)/2}\prod_{j=1}^N(u_j-c_1)\dots(u_j-c_{N-1})\\
\times\frac{\det\left[\dfrac1{(u_j\mid c_1,c_2,\dots)^{\mu_i+N-i}}\right]_{i,j=1}^N}
{V(u_1,\dots,u_N)}.
\end{multline}
\end{lemma}

\begin{proof} 
Let us show that
\begin{equation*}
\det\left[\dfrac1{(u_j\mid c_1,c_2,\dots)^{N-i}}\right]_{i,j=1}^N  =
(-1)^{N(N-1)/2}\,\frac{V(u_1,\dots,u_N)}{\prod_{j=1}^N(u_j-c_1)\dots(u_j-c_{N-1})},
\end{equation*}
then the lemma will follow. 

Write the left-hand side as $\det[A(i,j)]_{i,j=1}^N$ with
$$
A(i,j)=\frac1{(u_j\mid c_1,c_2,\dots)^{N-i}}.
$$
The product $A(i,j)(u_j-c_1)\dots(u_j-c_{N-1})$ is a
monic polynomial in $u_j$ of degree $i-1$, hence
$$
\det[A(i,j)]\prod_{j=1}^N(u_j-c_1)\dots(u_j-c_{N-1})=\det[u_j^{i-1}]=(-1)^{N(N-1)/2}
V(u_1,\dots,u_N),
$$
as desired.
\end{proof}

\begin{lemma}\label{lemma3}
The $\si$-functions in $N$ variables form a topological basis in the algebra
of symmetric power series in variables $u_1^{-1},\dots,u_N^{-1}$. 
\end{lemma}

\begin{proof}
Indeed, we have
$$
\si_{\mu\mid N}(u_1,\dots,u_N\mid c_1,c_2,\dots)=s_{\mu\mid N}(u_1^{-1},\dots,u_N^{-1})+\dots,
$$
where the dots denote higher degree terms in the variables $y_i:=u_i^{-1}$ (equivalently, lower degree terms in the variables $u_i$). 
\end{proof}

\begin{proposition}[Cauchy identity]\label{thm2.A}
For $N=1,2,\dots$ one has
\begin{multline}\label{eq.B}
\sum_{\mu\in\Y(N)}s_{\mu \mid N}(x_1,\dots,x_N\mid c_0,c_1,\dots)
\si_{\mu \mid N}(u_1,\dots,u_N\mid c_1,c_2,\dots)
\\
=\prod_{j=1}^N\frac{(u_j-c_0)\dots(u_j-c_{N-1})}{(u_j-x_1)\dots(u_j-x_N)},
\end{multline}
where both sides are regarded as power series in variables $x_1,\dots,x_N,u_1^{-1},\dots,u_N^{-1}$.
\end{proposition}

 This is particular case of Theorem 3.1 in Molev \cite{Molev}. However, matching definitions and notation requires some effort, so we give a short direct proof.

\begin{proof}
Examine first the simplest case $N=1$. Then the desired identity takes the
form
\begin{equation}\label{eq.C}
\sum_{m=0}^\infty
\frac{(x\mid c_0,c_1,\dots)^m}{(u\mid c_1,c_2,\dots)^m}=\frac{u-c_0}{u-x}.
\end{equation}
This is verified exactly as in Okounkov-Olshanski \cite[Theorem 12.1, first proof]{OO-AA}.

Consider now the general case. We use the well-known general identity
\begin{equation}\label{eq.D}
\sum_{\mu\in\Y(N)}\det[f_{\mu_r+N-r}(x_i)]_{r,i=1}^N\det[g_{\mu_r+N-r}(u_j)]_{r,j=1}^N=\det[h(i,j)]_{i,j=1}^N,
\end{equation}
where
$$
h(i,j):=\sum_{m=0}^\infty f_m(x_i)g_m(u_j).
$$
Here it is assumed that the infinite series make sense, which holds true in the concrete case that we need.

Now we set 
$$
f_m(x):=(x\mid c_0,c_1,\dots)^m, \qquad g_m(u):=\frac1{(u\mid c_1,c_2,\dots)^m}.
$$
Using \eqref{eq.A}, one can write the left-hand side of \eqref{eq.B} as
\begin{multline*}
\frac{(-1)^{N(N-1)/2}\prod_{j=1}^N(u_j-c_1)\dots(u_j-c_{N-1})}{V(x_1,\dots,x_N)V(u_1,\dots,u_N)}\\
\times\sum_{\mu\in\Y(N)}\det[f_{\mu_r+N-r}(x_i)]_{r,i=1}^N\det[g_{\mu_r+N-r}(u_j)]_{r,j=1}^N.
\end{multline*}
By \eqref{eq.D}, this is equal to
\begin{equation}\label{eq.E}
\frac{(-1)^{N(N-1)/2}\prod_{j=1}^N(u_j-c_1)\dots(u_j-c_{N-1})}{V(x_1,\dots,x_N)V(u_1,\dots,u_N)}\det[h(i,j)]_{i,j=1}^N.
\end{equation}

On the other hand, by \eqref{eq.C}
$$
h(i,j)=\frac{u_j-c_0}{u_j-x_i}.
$$
Hence
\begin{multline*}
\det[h(i,j)]=\prod_{j=1}^N(u_j-c_0)\cdot\det\left[\frac1{u_j-x_i}\right]\\
=(-1)^{N(N-1)/2}\prod_{j=1}^N(u_j-c_0)\cdot\prod_{j=1}^N\frac{V(x_1,\dots,x_N)V(u_1,\dots,u_N)}{\prod_{i,j=1}^N(u_j-x_i)}
\end{multline*}
(the Cauchy determinant formula).

Substituting this into \eqref{eq.E} we obtain the right-hand side of \eqref{eq.B}. 
\end{proof}

As a corollary, we obtain a determinantal formula for the $N$-variate dual functions with with $t=q$:

\begin{corollary}\label{cor4.A}
One has
\begin{equation}\label{eq.F}
\Idm_{\mu\mid N}(u_1,\dots,u_N;q,q)
=\si_{\mu\mid N}(u_1,\dots,u_N; q^{N-2}, q^{N-3}, q^{N-4},\dots).
\end{equation}
\end{corollary}

\begin{proof}
In the case $c_n=q^{N-n-1}$ the Cauchy identity \eqref{eq.B} takes the form
\begin{multline}\label{eq.G}
\sum_{\mu\in \Y(N)}I_{\mu\mid N}(x_1,\dots,x_N;q,q)\si_{\mu\mid N}(u_1,\dots,u_N; q^{N-2}, q^{N-3},q^{N-4},\dots)\\
=\prod_{j=1}^N\frac{(u_j-q^{N-1})(u_j-q^{N-2})\dots(u_j-1)}{(u_j-x_1)\dots(u_j-x_N)}=\prod_{j=1}^N\frac{(u_j^{-1};q)_N}{(1-x_1u_j^{-1})\dots(1-x_N u_j^{-1})}.
\end{multline}

Compare \eqref{eq.G} with the Cauchy identity \eqref{eq3.B}. When $t=q$, the right-hand side of \eqref{eq3.B} coincides with the right-hand side of \eqref{eq.G}. This implies \eqref{eq.F}.
\end{proof}

\section{Combinatorial formula}\label{sect5}

\subsection{Formulation of the theorem}
Recall that a \emph{horizontal strip} is a skew Young diagram $\mu/\nu$ containing at most one box in each column. The equivalent condition is that the coordinates of $\mu$ and $\nu$ \emph{interlace} in the sense that $\mu_i\ge\nu_i\ge\mu_{i+1}$ for $i=1,2,\dots$\,. In this case we write $\mu\succ\nu$ or $\nu\prec\mu$. 

To each horizontal strip $\mu\succ\nu$ we attach a univariate rational function given by
\begin{equation}\label{eq5.B}
\Wm_{\mu/\nu}(u;q,t):=\varphi_{\mu/\nu}(q,t)\frac{\prod_{i=1}^{\ell(\mu)}\prod_{j=1}^{\nu_i-\mu_{i+1}}(u-q^{-\nu_i+j-1}t^i)}{(u-q^{-1})\dots(u-q^{-\mu_1})},
\end{equation}
where the factor $\varphi_{\mu/\nu}(q,t)$ is defined in Macdonald \cite[Chapter VI, (6.24) and Example 2(a)]{M}. Note that the numerator is a monic polynomial of degree
$$
\sum_i(\nu_i-\mu_{i+1})=|\nu|-|\mu|+\mu_1,
$$ 
so that the whole fraction has degree $|\nu|-|\mu|\le0$ in $u$. 

Recall that by $\RTab(\mu,N)$ we denote the set of reverse Young tableaux of a given shape $\mu\in\Y(N)$ with values in $\{1,\dots,N\}$ (see their definition just before Proposition \ref{prop2.H}). Equivalently, a tableau $T\in\RTab(\mu,N)$ may be thought of as a chain
$$
T=(\mu(0)=\mu\succ\mu(1)\succ\mu(2)\succ\dots\succ \mu(N)=\varnothing).
$$ 
We also consider the set 
$$
\RTab(\mu):=\bigcup_{N=1}^\infty\RTab(\mu,N),
$$
whose elements may be regarded as infinite chains $(\mu(0)=\mu\succ\mu(1)\succ\mu(2)\succ\dots)$ such that $\mu(k)=\varnothing$ for  all $k$ large enough. 

\begin{theorem}[Combinatorial formula for modified dual functions]\label{thm5.A}
In the notation introduced above we have 
$$
\Idm_\mu(u_1,u_2,\dots;q,t)=\sum_{T\in\RTab(\mu)}\prod_{i=1}^\infty \Wm_{\mu(i-1)/\mu(i)}(u_i;q,t).
$$
\end{theorem}

Here are immediate corollaries.

\begin{corollary}\label{cor5.A}
For $\mu\in\Y(N)$ we have
\begin{equation}\label{combJ}
\Idm_{\mu\mid N}(u_1,\dots,u_N;q,t)=\sum_{T\in\RTab(\mu,N)}\prod_{i=1}^N \Wm_{\mu(i-1)/\mu(i)}(u_i;q,t).
\end{equation}
\end{corollary}

\begin{corollary}\label{cor5.B}
The $N$-variate functions $\Idm_{\mu\mid N}(u_1,\dots,u_N;q,t)$ are rational functions.
\end{corollary}

\subsection{A reformulation}

We define the \emph{skew dual functions} $\Id_{\mu/\nu}(\ccdot;q,t)$ in the standard way, by means of the relation
$$
\Id_\mu(Y'\cup Y;q,t)=\sum_{\nu\in\Y} \Id_{\mu/\nu}(Y';q,t)\Id_\nu(Y;q,t),
$$
where $Y$ and $Y'$ are two collections of variables.

Usually skew functions are indexed by skew diagrams, but at this moment we do not know yet whether $\nu$ must be contained in $\mu$. This is indeed true, but not evident from the definition. 

In what follow we assume that $Y'$ is reduced to a single variable $y$.

Recall that the univariate skew $Q$-functions have the form
\begin{equation}\label{eq5.phi}
Q_{\mu/\nu}(y;q,t)=\begin{cases}\varphi_{\mu/\nu}(q,t) y^{|\mu|-|\nu|}, & \mu\succ\nu,\\
0, &\text{otherwise},
\end{cases}
\end{equation}
where explicit expressions for the factor $\varphi_{\mu/\nu}(q,t)$ are given in Macdonald \cite[Chapter VI, (6.24) and Example 2(a)]{M}.

\begin{theorem}\label{thm5.B}
The univariate skew dual functions $\Id_{\mu/\nu}(y;q,t)$ have the following form{\rm:}

{\rm(i)} $\Id_{\mu/\nu}(y;q,t)$ vanishes unless $\mu\succ\nu${\rm;}

{\rm(ii)} if $\mu\succ\nu$, then
\begin{equation*}
\Id_{\mu/\nu}(y;q,t)=\varphi_{\mu/\nu}(q,t) y^{|\mu|-|\nu|}\prod_{i=1}^\infty\frac{(yq^{-\nu_i}t^i;q)_\infty}{(yq^{-\mu_i}t^{i-1};q)_\infty} 
\end{equation*}
\end{theorem}

The proof is given in the next two subsections. 

\begin{proof}[Derivation of Theorem \ref{thm5.A} from Theorem \ref{thm5.B}]
By virtue of Theorem \ref{thm5.B} and Definition \ref{def3.B}, it suffices to check that if $\mu\succ\nu$ and $\Id_{\mu/\nu}(y;q,t)$ is given by the formula above, then  
$$
(u^{-1};q)_\infty \Id_{\mu/\nu}(u^{-1};q,t)=\Wm_{\mu/\nu}(u;q,t).
$$
The left-hand side equals
\begin{multline*}
\varphi_{\mu/\nu}(q,t)u^{|\nu|-|\mu|}\frac{(u^{-1};q)_\infty}{(u^{-1}q^{-\mu_1};q)_\infty}\, \prod_{i=1}^\infty\frac{(u^{-1}q^{-\nu_i}t^i;q)_\infty}{(u^{-1}q^{-\mu_{i+1}}t^i;q)_\infty} \\
=\varphi_{\mu/\nu}(q,t)\frac{\prod\limits_{i=1}^\infty(u^{-1}q^{-\nu_i}t^i;q)_{\nu_i-\mu_{i+1}}}{u^{|\mu|-|\nu|}(u^{-1}q^{-\mu_1};q)_{\mu_1}}=\Wm_{\mu/\nu}(u;q,t).
\end{multline*}
\end{proof}

\subsection{Proof of Theorem \ref{thm5.B}: reduction to a Pieri-type formula}

We proceed to the proof of Theorem \ref{thm5.B}.  The first step, given by the next lemma, is a standard argument. It shows that, by duality, the problem reduces to finding a Pieri-type formula for the interpolation functions. 

\begin{lemma}
$\Id_{\mu/\nu}(y;q,t)$ equals the coefficient  $c(\nu,\mu;y)$ in the expansion
\begin{equation}\label{eq5.D}
I_{\nu}(X;q,t)\cdot\prod_{i=1}^\infty\frac{(x_iyt;q)_\infty}{(x_iy;q)_\infty}=\sum_{\mu\in\Y}c(\nu,\mu;y) I_{\mu}(X;q,t),
\end{equation}
where $X=(x_1,x_2,\dots)$ and both sides are viewed as elements of\/ $\Sym_\FF[[y]]$. 
\end{lemma}

\begin{proof}
Write the Cauchy identity \eqref{eq3.C} as   
$$
\sum_\mu I_\mu(X;q,t)H_\mu(Y;q,t)=\Pi(X;Y;q,t),
$$
where $Y$ is an arbitrary collection of variables and
$$
\Pi(X;Y;q,t):=\prod_{i=1}^\infty\prod_{y\in Y}\frac{(x_iyt;q)_\infty}{(x_iy;q)_\infty}.
$$
Then we replace $Y$ by $Y\cup y$ and transform the left-hand side in two ways.

First, we may write
\begin{gather*}
\sum_\mu I_\mu(X;q,t)\Id_\mu(Y\cup y;q,t)=\Pi(X,Y\cup y;q,t)=\Pi(X;Y;q,t)\Pi(X;y;q,t)\\
=\left(\sum_\nu I_\nu(X;q,t)\Id_\nu(Y;q,t)\right)\Pi(X;y;q,t)=\sum_\nu \left(I_\nu(X;q,t)\Pi(X;y;q,t)\right)\Id_\nu(Y;q,t).
\end{gather*}

On the other hand, the same expression can be rewritten as 
\begin{gather*}
\sum_\mu I_\mu(X;q,t)\Id_\mu(Y\cup y;q,t)=\sum_\mu I_\mu(X;q,t)\sum_\nu \Id_\nu(Y;q,t)\Id_{\mu/\nu}(y;q,t)\\
=\sum_\nu \left(\sum_\mu I_\mu(X;q,t)\Id_{\mu/\nu}(y;q,t)\right)\Id_\nu(Y;q,t).
\end{gather*}

Equating the coefficients of $\Id_\nu(Y;q,t)$ gives the desired equality.
\end{proof}

For $\nu,\mu\in\Y$ and $N$ large enough (so that $N\ge\max(\ell(\nu), \ell(\mu))$) we denote by $c_N(\nu,\mu;y)$  the  coefficients in the expansion
\begin{equation}\label{eq5.c_N}
I_{\nu\mid N}(x_1,\dots,x_N;q,t)\cdot\prod_{i=1}^N\frac{(x_iyt;q)_\infty}{(x_iy;q)_\infty}=\sum_{\mu\in\Y(N)}c_N(\nu,\mu;y) I_{\mu\mid N}(x_1,\dots,x_N;q,t).
\end{equation}

Setting $x_i=t^{i-1}$ for $i\ge N+1$ in \eqref{eq5.D} we obtain the relation
$$
c(\nu,\mu;y)=\frac{c_N(\nu,\mu;y)}{(yt^N;q)_\infty}, \qquad N\ge\max\{\ell(\nu),\ell(\mu)\}.
$$
Thus, the problem is reduced to computing the coefficients $c_N(\nu,\mu;y)$. This will be achieved  by adapting  Rains' argument in \cite[Theorem 3.13]{Rains}. Because Rains' exposition is very condensed,  I present a detailed proof.

\subsection{Proof of Theorem \ref{thm5.B}: computation of Pieri coefficients}
Let 
$$
f(x):=1+f_1 x+f_2 x^2+\dots
$$
be an arbitrary formal power series with constant term $1$. Denote by $\spec_f$ the specialization of the algebra of symmetric functions defined by
$$
\spec_f\left(1+\sum_{n=1}^\infty Q_{(n)}(\ccdot;q,t)x^n\right)=f(x),
$$
that is,
$$
\spec_f(Q_{(n)}(\ccdot;q,t))=f_n, \qquad n=1,2,\dots\,.
$$

\begin{lemma}\label{lemma5.A}
In this notation, for any diagram $\nu\in\Y$ one has
\begin{equation*}
P_\nu(x_1,x_2,\dots;q,t)\prod_{i=1}^\infty f(x_i)=\sum_{\mu: \, \mu\supseteq \nu}\spec_f(Q_{\mu/\nu}(\ccdot;q,t))P_\mu(x_1,x_2,\dots;q,t).
\end{equation*}
\end{lemma}

\begin{proof}
Let $X$ and $Z$ be two collections of formal variables. A standard argument (of the sort we have already used in the previous lemma) shows that
$$
P_\nu(X;q,t)\Pi(X;Z;q,t)=\sum_{\mu: \, \mu\supseteq \nu}Q_{\mu/\nu}(Z;q,t) P_\mu(X;q,t).
$$
Apply $\spec_f$ to the both sides, with respect to the variables $Z$: 
$$
P_\nu(X;q,t)\spec_f(\Pi(X;\ccdot;q,t))=\sum_{\mu: \, \mu\supseteq \nu}\spec_f(Q_{\mu/\nu}(\ccdot;q,t)) P_\mu(X;q,t).
$$
Observe that 
$$
\Pi(x;Z;q,t)=1+\sum_{n=1}^\infty Q_{(n)}(Z;q,t)x^n, \qquad \Pi(X;Z;q,t)=\prod_{x\in X}\Pi(x;Z;q,t).
$$  
It follows that
$$
\spec_f(\Pi(X;\ccdot;q,t))=\prod_{x\in X}f(x),
$$
which completes the proof.
\end{proof}

In the next lemma,  $\nu\in\Y(N)$, $m$ is a natural number, and $a$ is an indeterminate. Next, we set $\spec:=\spec_f$, where $f(x)=(x;q)_m$. Finally, we denote by $m^N+\nu$ the diagram
$$
(\nu_1+m,\,\nu_2+m,\,\dots,\,\nu_N+m)\in\Y(N).
$$

\begin{lemma}\label{lemma5.spec}
In the expansion
\begin{equation}\label{eq5.A}
I_{\nu\mid N}(x_1,\dots,x_N;q,t)\prod_{i=1}^N(x_i a;q)_m=\sum_{\mu\in\Y(N)}c_{N,m}(\nu,\mu;a) I_{\mu\mid N}(x_1,\dots,x_N;q,t),
\end{equation}
the coefficients $c_{N,m}(\nu,\mu;a)$ vanish unless $\nu\subseteq\mu\subseteq m^N+\nu$, and if this condition is satisfied, then
\begin{equation}\label{eq5.A1}
c_{N,m}(\nu,\mu;a)=\spec(Q_{\mu/\nu}(\ccdot;q,t))a^{|\mu|-|\nu|}\prod_{(i,j)\in (m^N+\nu)/\mu}(1-q^{m-j}t^{i-1}a).
\end{equation}
\end{lemma}

\begin{proof}
\emph{Step} 1. Let us show that the coefficient $c_{N,m}(\nu,\mu;a)$ vanishes unless $\mu\supseteq\nu$. Indeed, let $\la\in\Y(N)$ be a minimal (by inclusion) diagram such that $c_{N,m}(\nu,\la;a)\ne0$, and evaluate both sides of \eqref{eq5.A} at $X_N(\la)$. Then the right-hand side does not vanish, because it contains a single nonzero summand --- that with $\mu=\la$ (this follows from the extra vanishing property of the interpolation polynomials). Therefore,  the left-hand side does not vanish, too, whence $I_{\nu\mid N}(X_N(\la);q,t)\ne0$, which implies $\la\supseteq\nu$.  

Thus, we may assume that the sum in \eqref{eq5.A} is in fact taken over diagrams $\mu\supseteq\nu$, in particular, $|\mu|\ge|\nu|$.

\emph{Step} 2. Since $\prod_{i=1}^N(x_i a;q)_m$ is a polynomial in variable $a$ of degree $Nm$, the coefficient $c_{N,m}(\nu,\mu;a)$ is a polynomial in $a$ of degree at most $Nm$.

Let us replace each $x_i$ by $x_i r$ and also replace $a$ by $a/r$, where $r$ is a numeric parameter. Then we may write
\begin{multline*}
r^{-|\nu|}I_{\nu\mid N}(x_1r,\dots,x_Nr;q,t)\prod_{i=1}^N(x_i a;q)_m\\
=\sum_{\mu\in\Y(N)}r^{|\mu|-|\nu|}c_{N,m}(\nu,\mu;a/r) (r^{-|\mu|}I_{\mu\mid N}(x_1r,\dots,x_Nr;q,t).
\end{multline*}

As $r\to\infty$, we have
$$
r^{-|\nu|}I_{\nu\mid N}(x_1r,\dots,x_Nr;q,t)\to P_{\nu\mid N}(x_1,\dots,x_N;q,t)
$$
and
$$
r^{-|\mu|}I_{\mu\mid N}(x_1r,\dots,x_Nr;q,t)\to P_{\mu\mid N}(x_1,\dots,x_N;q,t),
$$
so that in the limit, we obtain the expansion
$$
P_{\nu\mid N}(x_1,\dots,x_N;q,t)\prod_{i=1}^N(x_i a;q)_m=\sum_{\mu\in\Y(N)} (\cdots)P_{\mu\mid N}(x_1,\dots,x_N;q,t)
$$
with some coefficients. More precisely, by the previous lemma, the coefficient of $P_{\mu\mid N}(x_1,\dots,x_N)$ is equal to $\spec(Q_{\mu/\nu}(\ccdot;q,t))a^{|\mu|-|\nu|}$.
It follows that 
$$
\lim_{r\to\infty}r^{|\mu|-|\nu|}c_{N,m}(\nu,\mu;a/r)=\spec(Q_{\mu/\nu}(\ccdot;q,t))a^{|\mu|-|\nu|}.
$$
Therefore, we can write
$$
c_{N,m}(\nu,\mu;a)=a^{|\mu|-|\nu|}\wt c_{N,m}(\nu,\mu;a),
$$
where $\wt c_{N,m}(\nu,\mu;a)$ is a polynomial in $a$ of degree at most $Nm-|\mu|+|\nu|$, with constant term  $\spec(Q_{\mu/\nu}(\ccdot;q,t))$.

\emph{Step} 3. We keep to the assumption that $\mu\in\Y(N)$ and $\mu\supseteq\nu$. Set
$$
d_\mu(a):=\prod_{i=1}^N(x_i a;q)_m\Bigg|_{(x_1,\dots,x_N)=X_N(\mu)}.
$$
That is, 
$$
d_\mu(a):=\prod_{i=1}^N(q^{-\mu_i}t^{i-1}a;q)_m=\prod_{i=1}^N\prod_{k=0}^{m-1}(1-q^{-\mu_i+k}t^{i-1}a).
$$
Observe that the roots of this polynomial are pairwise distinct (recall that we are working over $\FF$). 

Let $d_{\nu\mu}(a)$ denote the product of common factors of the polynomials $d_\nu(a)$ and $d_\mu(a)$.
We claim that
$$
d_{\nu\mu}(a)=\prod_{(i,j)\in m^N+\nu, \; (i,j)\notin \mu}(1-q^{m-j}t^{i-1}a).
$$
Indeed, in $d_\mu(a)$, for fixed index $i$, the exponent of $q$ ranges over the set
$$
\{-\mu_i,\dots,-\mu_i+m-1\}=\{m-j: \mu_i+1\le j\le \mu_i+m\}.
$$
Likewise, in $d_\nu(a)$, the corresponding set is
$$
\{m-j: \nu_i+1\le j\le \nu_i+m\}.
$$
As $\nu_i\le\mu_i$, the intersection of these two sets is
$$
\{m-j: \mu_i+1\le j\le \nu_i+m\},
$$
which proves the claim.

\emph{Step} 4. Let us prove that for each $\la\in\Y(N)$ with $\la\supseteq\nu$, the polynomial $c_{N,m}(\nu,\la;a)$ defined by \eqref{eq5.A} is divisible by $d_{\nu\la}(a)$. We use induction on $\la$. 

The base of induction is $\la=\nu$. Let us substitute $(x_1,\dots,x_N)=X_N(\nu)$ into \eqref{eq5.A}. On the right, $I_{\mu\mid N}(X_N(\nu);q,t)=0$ for all $\mu\ne\nu$, and we obtain
$$
I_{\nu\mid N}(X_N(\nu);q,t)d_{\nu}(a)=c_{N,m}(\nu,\nu;a)I_{\nu\mid N}(X_N(\nu);q,t),
$$
so that 
$$
c_{N,m}(\nu,\nu;a)=d_\nu(a)=d_{\nu\nu}(a).
$$

Next, by virtue of step 3, as $\la$ gets larger, the polynomial $d_{\nu\la}(a)$ gets smaller in the sense that it involves a smaller set of factors. Let us substitute $(x_1,\dots,x_N)=X_N(\la)$ into \eqref{eq5.A}. The left-hand side equals  $d_\la(a)$ multiplied by a constant factor. Hence the right-hand side must vanish at all roots of $d_\la(a)$. A fortiori, it must vanish at all roots of $d_{\nu\la}(a)$. However, at these roots, all the polynomials $d_{\nu\mu}(a)$ with $\mu\subset\la$ also vanish. Therefore, by the induction assumption, all the polynomials $c_{N,m}(\nu,\mu;a)$ with $\mu$ strictly contained in $\la$ vanish at these roots, too. 

On the other hand, we have $I_{\mu\mid N}(X_N(\la);q,t)=0$ unless $\mu\subseteq\la$. Thus, only $\mu=\la$ contributes, so the right-hand side is reduced to $c_{N,m}(\nu,\la;a)I_{\la\mid N}(X_N(\la);q,t)$. We conclude that $c_{N,m}(\nu,\la;a)$ is divisible by $d_{\nu\la}(a)$.

\emph{Step} 5. Let again $\mu\in\Y(N)$, $\mu\supseteq\nu$. Recall that $\deg c_{N,m}(\nu,\mu;a)\le Nm$. On the other hand, 
$$
\deg d_{\nu\mu}(a)\ge Nm-(|\mu|-|\nu|)
$$
and the equality holds only when $\mu\subseteq m^N+\nu$. 

Since $c_{N,m}(\nu,\mu;a)$ is also divisible by $a^{|\mu|-|\nu|}$, we see that only diagrams $\mu$ contained in $m^N+\nu$ may contribute to \eqref{eq5.A}, and for these $\mu$'s, we have
$$
c_{N,m}(\nu,\mu)=\const a^{|\mu|-|\nu|}d_{\nu\mu}(a).
$$
As the constant term of $d_{\nu\mu}(a)$ equals $1$, we finally obtain the desired formula.
\end{proof}

The next lemma is a generalization of the previous one.
Let $a$ and $b$ be formal variables, and  let $\wt\spec:\Sym_\FF\to \FF[a,b]$ be the specialization  $\spec_f$ with
$$
f(x):= \frac{(xa;q)_\infty}{(xb;q)_\infty}.
$$

\begin{lemma}
In this notation, we have{\rm:}

{\rm(i)} The expansion
\begin{equation}\label{eq5.B0}
I_{\nu\mid N}(x_1,\dots,x_N;q,t)\prod_{i=1}^N\frac{(x_i a;q)_\infty}{(x_ib;q)_\infty}=\sum_{\mu\in\Y(N)}c_N(\nu,\mu;a,b) I_{\mu\mid N}(x_1,\dots,x_N;q,t)
\end{equation}
makes sense as an equality in $\Sym_\FF(N)[[a,b]]$.

{\rm(ii)} In this expansion, the coefficients vanish unless $\mu\supseteq\nu$, and if this condition is satisfied, then
\begin{equation}\label{eq5.B1}
c_N(\nu,\mu;a,b)=\wt\spec(Q_{\mu/\nu}(\ccdot;q,t))\prod_{i=1}^N\frac{(q^{-\nu_i}t^{i-1}a;q)_\infty}{(q^{-\mu_i}t^{i-1}b;q)_\infty}.
\end{equation}
\end{lemma}

\begin{proof}
(i) This  is obvious. 

(ii) Observe that if $F(a,b)$ is a formal power series in variables $a$ an $b$,  such that $F(a,a q^m)=0$ for $m=1,2,\dots$, then $F(a,b)=0$.  Indeed, write
$$
F(a,b)=\sum_{k,\ell=0}^\infty F_{k\ell}a^k b^\ell.
$$
The condition on $F$ means that for each $n=0,1,2,\dots$ one has
$$
\sum_{k+m \ell=n}F_{k\ell}=0, \qquad m=1,2,\dots\,.
$$
Taking $m>n$ we see that $F_{n0}=0$ for all $n$. Then divide the series by $b$ and repeat the same argument, and so on. In this way we obtain that all coefficients vanish. 

This argument shows that it suffices to check claim (ii) in the particular case $b=a q^m$. Let us show that then (ii) reduces to the claim of Lemma \ref{lemma5.spec}.

Indeed, if $b=aq^m$, then
$$
\frac{(xa;q)_\infty}{(xb;q)_\infty}=\frac{(xa;q)_\infty}{(xa^m;q)_\infty}=(xa;q)_m.
$$
It follows that the left-hand side of \eqref{eq5.B0} turns into the left-hand side of \eqref{eq5.A}. 

Next, recall that the specialization $\spec$ in Lemma \ref{lemma5.spec} corresponds to $f(x)=(x;q)_m$. Since the specialization $\wt\spec$ with $b=a^m$ corresponds to $f(x)=(xa;q)_m$, it follows that  
$$
\wt\spec(Q_{\mu/\nu}(\ccdot;q,t))=\spec(Q_{\mu/\nu}(\ccdot;q,t)) a^{|\mu|-|\nu|}, 
$$
because $Q_{\mu/\nu}(\ccdot;q,t)$ is homogeneous of degree $|\mu|-|\nu|$. 

Next, observe that in the case $b=aq^m$ the product on the right-hand side of \eqref{eq5.B1} turns into the product on the right-hand side of \eqref{eq5.A1}. Therefore,  \eqref{eq5.B1} turns into \eqref{eq5.A1}.

The final observation is that the quantity $\wt\spec(Q_{\mu/\nu}(\ccdot;q,t))$ with $b=aq^m$ automatically vanishes unless $\mu\subseteq m^N+\nu$. Indeed, from the proof of Lemma \ref{lemma5.A} and the fact that the polynomial $f(x)=(xa;q)_m$ has degree $m$ it is seen that $\wt\spec(Q_{\mu/\nu}(\ccdot;q,t))$ vanishes unless $\mu$ can be obtained from $\nu$ by appending a horizontal strip of length at most $m$ (here we also use the Pieri rule). Therefore, $\mu_i\le \nu_i+m$ for all $i$, so that $\mu\subseteq m^N+\nu$. 

Thus, in the case $b=aq^m$ the desired equality \eqref{eq5.B0} coincides with \eqref{eq5.A}, which completes the proof.  
\end{proof}

Now we are in a position to find the coefficients $c_N(\nu,\mu;y)$ of the expansion \eqref{eq5.c_N}. 

\begin{lemma}
Let $\nu,\mu\in\Y$ and $N\ge \max\{\ell(\nu),\ell(\mu)\}$. 

{\rm(i)} The coefficient $c_N(\nu,\mu;y)$ vanishes unless $\mu\succ\nu$. 

{\rm(ii)} If this condition holds, then
$$
c_N(\nu,\mu;y)=\varphi_{\mu/\nu}(q,t) y^{|\mu|-|\nu|}\prod_{i=1}^N\frac{(yq^{-\nu_i}t^i;q)_\infty}{(yq^{-\mu_i}t^{i-1};q)_\infty},
$$
where the factors $\varphi_{\mu/\nu}(q,t)$ are defined by \eqref{eq5.phi}.
\end{lemma}

\begin{proof} 
The coefficients $c_N(\nu,\mu;y)$ are a particular case of the coefficients $c_N(\nu,\mu;a,b)$ corresponding to $a=yt$, $b=y$. The previous lemma gives us an expression for the coefficients $c_N(\nu,\mu;a,b)$. It involves the specialization $\wt\spec$ depending on $a$ and $b$.  In the case $a=yt$, $b=y$, this specialization takes the form defined by
$$
\Pi(x;\ccdot;q,t)=1+\sum_{n=1}^\infty Q_{(n)}(\ccdot;q,t)x^n\mapsto \frac{(xyt;q)_\infty}{(xy;q)_\infty}.
$$ 
This means that it is simply evaluation at the point $(y,0,0,\dots)$. By virtue of \eqref{eq5.phi} this immediately implies both claims. 
\end{proof}

From the formula for $c_N(\nu,\mu;y)$ it is seen that the ratio $c_N(\nu,\mu;y)/(yt^N;q)_\infty$ does not depend on $N$ (as it should be) and is equal to 
$$
\varphi_{\mu/\nu}(q,t) y^{|\mu|-|\nu|}\prod_{i=1}^\infty\frac{(yq^{-\nu_i}t^i;q)_\infty}{(yq^{-\mu_i}t^{i-1};q)_\infty}.
$$
We have obtained the desired expression for $c(\nu,\mu;y)$. This completes the proof of Theorem \ref{thm5.B}.

\section{Difference operators: special case $t=q$}\label{sect6}

The present section serves as a preparation to the next one. We fix a positive integer $N$ and assume $t=q$. 

We use the standard notation $T_{q,x}$ for the $q$-shift operator acting on a variable $x$ (see Macdonald \cite[ch. VI, (3.1)]{M}): for a test function $f(x)$,
$$
(T_{q,x}f)(x)=f(xq).
$$
Let $z$ be an auxiliary variable and let $L_x$ denote the following $q$-difference operator acting on functions in $x$:
$$
L_x:=x^{-1}\big((xq^{1-N}-1)zT_{q,x}+x+z\big).
$$
Note that $L_x$ preserves the space of polynomials. 

Next, we set
$$
\mathcal D_N(z;q)=\frac1{V(x_1,\dots,x_N)}\circ\prod_{j=1}^N L_{x_j}\circ V(x_1,\dots,x_N).
$$
That is, the operator on the right is the composition of three operators: multiplication by the Vandemonde $\prod_{i<j}(x_i-x_j)$, the product of partial $q$-difference operators $L_{x_j}$, and division by the Vandermonde. 

Recall (see Section \ref{sect4}) that 
\begin{multline*}
I_{\mu\mid N}(x_1,\dots,x_N; q,q)=s_{\mu\mid N}(x_1,\dots,x_N\mid q^{N-1}, q^{N-2},\dots)\\
=\frac{\det\left[(x_j-q^{N-1})\dots(x_j-q^{-\mu_i+i})\right]_{i=1}^N}{V(x_1,\dots,x_N)}.
\end{multline*}

\begin{lemma}\label{lemma6.A}
One has
$$
\mathcal D_N(z;q)\big(I_{\mu\mid N}(x_1,\dots,x_N;q,q)\big)=\prod_{i=1}^N(1+q^{\mu_i+1-i}z)\cdot I_{\mu\mid N}(x_1,\dots,x_N;q,q).
$$
\end{lemma}

\begin{proof}
Set 
$$
f_m(x)=\begin{cases} (x-q^{N-1})\dots(x-q^{-m}), &  m\in\Z, \, m\ge 1-N,\\
1, & m=-N. \end{cases}  
$$
Then 
$$
T_{q,x} f_m=\frac{x q^{N+m}-q^{N-1}}{x-q^{N-1}}f_m.
$$
It follows that
$$
L_x f_m=(1+q^{m+1}z)f_m,
$$
which in turn implies the desired result.
\end{proof}

Next, recall (see Corollary \ref{cor4.A} and Lemma \ref{lemma4.B}) that
\begin{multline*}
\Idm_{\mu\mid N}(u_1,\dots,u_N;q,q)
=\si_{\mu\mid N}(u_1,\dots,u_N\mid q^{N-2}, q^{N-3},\dots)\\
=(-1)^{N(N-1)/2}\frac{\det\left[\dfrac{(u_j-q^{N-2})\dots(u_j-1)}{(u_j-q^{N-2})\dots(u_j-q^{-\mu_i+i-1})}\right]_{i=1}^N}{V(u_1,\dots,u_N)}
\end{multline*}

Let us set
$$
\wh{\mathcal D}_N(z;q):=\frac1{V(u_1,\dots,u_N)}\circ\prod_{j=1}^N\wh L_{u_j}\circ V(u_1,\dots,u_N),
$$
where
$$
\wh L_u:=u^{-1}\big((u-1)z T^{-1}_{q,u}+u+z\big).
$$

The next lemma is similar to the previous one.

\begin{lemma}\label{lemma6.B}
One has
\begin{equation*}
\wh{\mathcal D}_N(z;q)\left( \Idm_{\mu\mid N}(u_1,\dots,u_N;q,q)\right)
= \prod_{i=1}^N(1+q^{\mu_i+1-i}z)\cdot  \Idm_{\mu\mid N}(u_1,\dots,u_N;q,q).
\end{equation*}
\end{lemma}

\begin{proof}
Set 
$$
\wh f_m(u)=\begin{cases} \dfrac{(u-q^{N-2})\dots(u-1)}{(u-q^{N-2})\dots(u-q^{-m-1})}, &  m\in\Z, \, m\ge 1-N,\\
1, & m=-N. \end{cases}  
$$
Then 
$$
T^{-1}_{q,u} \wh f_m=\frac{u q^{m+1}-1}{u-1}\wh f_m.
$$
It follows that 
$$
\wh L_u \wh f_m=(1+q^{m+1}z)\wh f_m,
$$
which in turn implies the desired result.
\end{proof}

\begin{corollary}\label{cor6.A}
One has
$$
\mathcal D_N(z,q)\left(\prod_{j=1}^N\frac{(u_j-1)\dots(u_j-q^{N-2})}{(u_j-x_1)\dots(u_j-x_N)}\right) 
=\wh{\mathcal D}_N(z,q)\left(\prod_{j=1}^N\frac{(u_j-1)\dots(u_j-q^{N-2})}{(u_j-x_1)\dots(u_j-x_N)}\right) 
$$
\end{corollary}

\begin{proof}
The Cauchy identity \eqref{eq.B} can be written in our case as
$$
\sum_{\mu\in\Y(N)}I_{\mu\mid N}(x_1,\dots,x_N;q,q)\Idm_{\mu\mid N}(u_1,\dots,u_N;q,q)=\prod_{j=1}^N\frac{(u_j-1)\dots(u_j-q^{N-2})}{(u_j-x_1)\dots(u_j-x_N)}.
$$
Combining this with Lemma \ref{lemma6.A} and Lemma \ref{lemma6.B} gives the desired equality.
\end{proof}

\begin{lemma}\label{lemma6.C}
The operators $\mathcal D_N(z;q)$ and $\wh{\mathcal D}_N(z,q)$ can be written in the following form
\begin{gather}
\mathcal D_N(z,q)=\frac1{V(x_1,\dots,x_N)} \det\bigg[x_j^{-1}\big\{(x_jq^{1-N}-1)x_j^{N-i}q^{N-i}zT_{q,x_j}+(x_j+z)x_j^{N-i}\big\}\bigg]_{i,j=1}^N   \label{eq6.A}\\
\wh{\mathcal D}_N(z,q)=\frac1{V(u_1,\dots,u_N)}\det\bigg[u_j^{-1}\big\{(u_j-1)u_j^{N-i}q^{i-N}zT^{-1}_{q,u_j}+(u_j+z)u_j^{N-i}\big\}\bigg]_{i,j=1}^N. \label{eq6.B}
\end{gather}
\end{lemma}

\begin{proof}
Writing $V(x_1,\dots,x_N)=\det[x_j^{N-i}]$ we obtain
$$
\mathcal D_N(z,q)=\frac1{V(x_1,\dots,x_N)}\det\big[L_{x_j}\circ x_j^{N-i}\big]_{i,j=1}^N,
$$
and then we use the relation $T_{q,x}\circ x^m=x^mq^m T_{q,x}$. This gives the first formula.
The second formula is checked similarly, using the relation $T^{-1}_{q,u}\circ u^m=u^mq^{-m} T^{-1}_{q,u}$.
\end{proof}

\section{Difference operators: general case}\label{sect7}

\subsection{$q$-Difference equations for polynomials $I_{\mu\mid N}(\ccdot;q,t)$}
The following operator is obtained from the operator $\mathcal D_N(z,q)$ (see \eqref{eq6.A}) be replacing $q$ with $t$ in the coefficients; the $q$-shifts $T_{q,x_j}$ remain intact:
\begin{multline}\label{eq7.A}
D_N(z;q,t)=\frac1{V(x_1,\dots,x_N)} \det\bigg[x_j^{-1}\big\{(x_jt^{1-N}-1)x_j^{N-i}t^{N-i}zT_{q,x_j}+(x_j+z)x_j^{N-i}\big\}\bigg]_{i,j=1}^N
\end{multline}

\begin{theorem}\label{thm7.A}
Let $\mu\in\Y(N)$. One has
\begin{equation}\label{eq7.A1}
D_N(z;q,t) \big(I_{\mu\mid N}(x_1,\dots,x_N;q,t)\big)=\prod_{j=1}^N(1+q^{\mu_i}t^{1-i}z)\cdot I_{\mu\mid N}(x_1,\dots,x_N;q,t).
\end{equation}
\end{theorem}

This result is a reformulation of the theorem proved in Okounkov \cite[\S 3]{Ok-MRL}. That theorem is a refinement of results from Knop \cite{Knop} and Sahi \cite{Sahi2}, which in turn are analogues of \cite[Theorem 4.4]{KnopSahi}. 

We need a lemma.

\begin{lemma}\label{lemma7.A}
Let $A$ range over the set of all subsets of $\{1,\dots,N\}$ and $A^c:=\{1,\dots,N\}\setminus A$. One has
$$
D_N(z;q,t)=\sum_A C_A z^{|A|}T_{q,A,X}
$$
where $X$ denotes the $N$-tuple of variables $x_1,\dots,x_N$,
$$
T_{q,A,X}:=\prod_{a\in A}T_{q,x_a},
$$
and
$$
C_A=C_A(X;t):=t^{|A|(|A|-1)/2}\cdot \prod_{a\in A}\frac{x_at^{1-N}-1}{x_a}\cdot\prod_{b\in A^c}\frac{x_b+z}{x_b}\cdot\prod_{a\in A}\prod_{b\in A^c}\frac{x_at-x_b}{x_a-x_b}.
$$
\end{lemma}

\begin{proof}
Let $S_N$ denote the group of permutations of $\{1,\dots,N\}$. The determinant on the right-hand side of \eqref{eq7.A} equals
\begin{gather*}
\sum_{\si\in S_N}\sgn(\si)\prod_{j=1}^N \bigg[x_j^{-1}\big\{(x_jt^{1-N}-1)x_j^{N-\si(j)}t^{N-\si(j)}zT_{q,x_j}+(x_j+z)x_j^{N-\si(j)}\big\}\bigg
] \\
=\sum_{A}\left(\prod_{a\in A}\frac{x_at^{1-N}-1}{x_a}\prod_{b\in A^c}\frac{x_b+z}{x_j}\right)\left(\sum_{\si\in S_N}\sgn(\si)\prod_{a\in A}(x_a t)^{N-\si(a)}\prod_{b\in A^c}
x_b^{N-\si(b)}\right) z^{|A|} T_{q,A,X}.
\end{gather*}

Given $A$, the sum over $\si\in S_N$ taken in the parentheses can be written as $V(\wt x_1,\dots,\wt x_N)$, where
$$
\wt x_j:= \begin{cases} x_jt, & j\in A, \\ x_j, & j\in A^c. \end{cases}
$$
Next, observe that
$$
\frac{V(\wt x_1,\dots,\wt x_N)}{V(x_1,\dots,x_N)}=t^{|A|(|A|-1)/2}\prod_{a\in A}\prod_{b\in A^c}\frac{x_at-x_b}{x_a-x_b}.
$$
This completes the proof.
\end{proof}

\begin{proof}[Proof of Theorem \ref{thm7.A}]
We will show that the function $D_N(z,q,t) \big(I_{\mu\mid N}(x_1,\dots,x_N;q,t)\big)$ possesses the following properties:

(1) it is a symmetric polynomial of degree $\le |\mu|$;

(2) it vanishes at $X_N(\la)$ for any $\la\in\Y(N)$ such that $|\la|\le|\mu|$ and $\la\ne\mu$;

(3) its value at $X_N(\mu)$ equals $\prod_{i=1}^N(1+q^{\mu_i}t^{1-i}z)\cdot I_{\mu\mid N}(X_N(\mu);q,t)$. 

Once this is done, the theorem will follow from the characterization of the interpolation polynomials.

\emph{Proof of} (1). This is obvious, because the operator $D_N(z;q,t)$ preserves the space of symmetric polynomials and does not rise degree.

\emph{Proof of} (2). Let $|\la|\le|\mu|$ and $\la\ne\mu$. We will prove a stronger claim: 
$C_A T_{q,A,X} \big(I_{\mu\mid N}(X;q,t)\big)$ vanishes at $X=X_N(\la)$ for every $A$. 

Indeed, set
$$
\de(A):=(\de_1(A),\dots,\de_N(A))\in\{0,1\}^N, \qquad \de_i(A):=\begin{cases} 1, & i\in A, \\ 0, & i\in A^c\end{cases}.
$$ 
Then
$$
C_A T_{q,A,X} \big(I_{\mu\mid N}(X;q,t)\big)\big|_{X=X_N(\la)}=C_A(X_N(\la)) I_{\mu\mid N}(X_N(\la-\de(A));q,t),
$$
where $X_N(\la-\de(A))$ denotes the vector whose $i$th coordinate equals $q^{-\la_i+\de_i(A)}t^{i-1}$, $i=1,\dots,N$.

If $\la-\de(A)$ is a partition, then $I_{\mu\mid N}(X_N(\la-\de(A));q,t)=0$. 

If $\la-\de(A)$ is not a partition, then there are two possible cases: 

$\bullet$ there exists an index $i\in\{1,\dots,N-1\}$ such $i\in A$, $i+1\in A^c$, and $\la_i=\la_{i+1}$;

$\bullet$ there is no such index, but $\la_N=0$ and $N\in A$.

In the first case, $C_A(X_N(\la))$ vanishes because it contains the factor $x_it-x_{i+1}$ that vanishes under the substitution $x_i=q^{-\la_i}t^{i-1}$, $x_{i+1}=q^{-\la_{i+1}}t^i$ as $\la_i=\la_{i+1}$. 

In the second case, $C_A(X_N(\la))$ vanishes because it contains the factor $x_Nt^{1-N}-1$ that vanishes under the substitution $x_N=q^{-\la_N}t^{N-1}$ as $\la_N=0$.  

\emph{Proof of} (3). The above argument also shows that $C_A T_{q,A,X} \big(I_{\mu\mid N}(X;q,t)\big)$ vanishes at $X=X_N(\mu)$ for every $A\ne\varnothing$. It remains to investigate the case of $A=\varnothing$. The operator $T_\varnothing$ is the identity operator, so we are left with the operator of multiplication by the function 
$C_\varnothing$, which is very simple:
$$
C_\varnothing(X)=\prod_{j=1}^N\frac{x_j+z}{x_j}=\prod_{j=1}^N(1+x_j^{-1}z).
$$
Its value at $X=X_N(\mu)$ is equal to $\prod_{j=1}^N(1+q^{\mu_i}t^{1-i}z)$, which completes the proof.
\end{proof}

We may write $D_N(z;q,t)$ as a polynomial in $z$ of degree $N$, with operator coefficients and constant term $1$:
$$
D_N(z;q,t)=1+\sum_{r=1}^N D^r_N(q,t) z^r.
$$
Here $D^r_N(q,t)$ is a $q$-difference operator of order $r$. 

\begin{example}\label{example7.A}
We have
$$
D^1_N(q,t)=\sum_{j=1}^N \frac{x_jt^{1-N}-1}{x_j}\prod_{k:\, k\ne j}\frac{x_jt-x_k}{x_j-x_k}\, T_{q,x_j}+ \sum_{j=1}^N\frac1{x_j}.
$$
Indeed, this is seen at once from Lemma \ref{lemma7.A}.
\end{example}

The next result is a direct consequence of Theorem \ref{thm7.A}, cf. Macdonald \cite[Ch. VI, (4.16)]{M}.

\begin{corollary}
The difference operators $D^1_N(q,t),\dots,D^N_N(q;t)$ commute with each other.
\end{corollary}

\subsection{$q$-Difference equations for modified dual functions $\Idm_{\mu\mid N}(\ccdot;q,t)$}

We define the $q$-difference operator $\wh D_N(z;q,t)$ in the same way as we did for $D_N(z;q,t)$: we take the expression for $\wh{\mathcal D}_N(z;q)$ and replace $q$ with $t$ in the coefficients:
$$
\wh D_N(z;q,t):=\frac1{V(u_1,\dots,u_N)}\det\bigg[u_j^{-1}\big\{(u_j-1)u_j^{N-i}t^{i-N}zT^{-1}_{q,u_j}+(u_j+z)u_j^{N-i}\big\}\bigg]_{i,j=1}^N.
$$

\begin{theorem}
For $N=1,2,\dots$ and $\mu\in\Y(N)$, 
\begin{equation}\label{eq7.A2}
\wh D_N(z;q,t)\big(\Idm_{\mu\mid N}(u_1,\dots,u_N;q,t)\big)=\prod_{i=1}^N(1+q^{\mu_i}t^{1-i}z) \cdot\Idm_{\mu\mid N}(u_1,\dots,u_N;q,t).
\end{equation}
\end{theorem} 

\begin{proof}
We apply a trick from Macdonald's book: see \cite[Chapter VI, proof of (3.12)]{M}. Consider the Cauchy identity (see \eqref{eq3.B})
$$
\sum_{\mu\in\Y(N)}I_{\mu\mid N}(x_1,\dots,x_N;q,t)\Idm_{\mu\mid N}(u_1,\dots,u_N;q,t)=\prod_{i,j=1}^N\frac{(x_iu_j^{-1}t;q)_\infty}{(x_iu_j^{-1};q)_\infty}\cdot \prod_{j=1}^N\frac{(u_j^{-1};q)_\infty}{(u_j^{-1}t^N;q)_\infty}
$$
and introduce a notation for the right-hand side:
\begin{equation}\label{eq7.Pi}
\wt\Pi(\ccdot;q,t)=\wt\Pi(x_1,\dots,x_N,u_1,\dots,u_N;q,t):=\prod_{i,j=1}^N\frac{(x_iu_j^{-1}t;q)_\infty}{(x_iu_j^{-1};q)_\infty}\cdot \prod_{j=1}^N\frac{(u_j^{-1};q)_\infty}{(u_j^{-1}t^N;q)_\infty}.
\end{equation}

By virtue of Theorem \ref{thm7.A}, it suffices to prove the relation
\begin{equation}\label{eq7.B}
D_N(z;q,t)\wt\Pi(\ccdot;q,t)\\
=\wh D_N(z;q,t)\wt\Pi(\ccdot;q,t).
\end{equation}

The key observation is that 
$$
T_{q,x_j}\big(\wt\Pi(x_1,\dots,x_N,u_1,\dots,u_N;q,t)\big)=(\cdots)\wt\Pi(x_1,\dots,x_N,u_1,\dots,u_N;q,t),
$$
where $(\cdots)$ is some factor that \emph{does not depend on $q$}. It follows that
$$
D_N(z;q,t)\big(\wt\Pi(x_1,\dots,x_N,u_1,\dots,u_N;q,t)\big)=(\cdots)\wt\Pi(x_1,\dots,x_N,u_1,\dots,u_N;q,t),
$$
where $(\cdots)$ is some expression that does not depend on $q$.

Likewise,
$$
\wh D_N(z;q,t)\big(\wt\Pi(x_1,\dots,x_N,u_1,\dots,u_N;q,t)\big)=(\cdots)\wt\Pi(x_1,\dots,x_N,u_1,\dots,u_N;q,t),
$$
where $(\cdots)$ is some expression that does not depend on $q$.

Therefore, it suffices to prove the desired relation \eqref{eq7.B} for $q=t$. But then it coincides with the formula of Corollary \ref{cor6.A}, where one should replace $q$ with $t$. 
\end{proof}

By analogy with the operators $D^r_N(q,t)$ we introduce operators $\wh D^r_N(q,t)$ as the coefficients in the expansion of $\wh D_N(z;q,t)$ in powers of $z$:
$$
\wh D_N(z;q,t)=1+\sum_{r=1}^N \wh D^r_N(q,t)z^r.
$$

\begin{corollary}
The operators $\wh D^1_N(q,t),\dots,\wh D^N_N(q,t)$ commute with each other.
\end{corollary}

\begin{example}\label{example7.B}
We have
$$
\wh D^1_N(q,t)=\sum_{j=1}^N \frac{u_j-1}{u_j}\prod_{k\ne j}\frac{u_jt^{-1}-u_k}{u_j-u_k}\, T^{-1}_{q,u_j}+ \sum_{j=1}^N\frac1{u_j}.
$$
Indeed, this is seen at once from the next lemma.
\end{example}

\begin{lemma}\label{lemma7.B}
Let, as before,  $A$ range over the set of all subsets of $\{1,\dots,N\}$ and $A^c:=\{1,\dots,N\}\setminus A$. We have
$$
\wh D_N(z;q,t)=\sum_A \wh C_A z^{|A|}T^{-1}_{q,A,U}
$$
where $U$ denotes the $N$-tuple of variables $u_1,\dots,u_N$,
$$
T^{-1}_{q,A,U}:=\prod_{a\in A}T^{-1}_{q,u_a},
$$
and
$$
\wh C_A=\wh C_A(U;t):=t^{-|A|(|A|-1)/2}\cdot \prod_{a\in A}\frac{u_a-1}{u_a}\cdot\prod_{b\in A^c}\frac{u_b+z}{u_b}\cdot\prod_{a\in A}\prod_{b\in A^c}\frac{u_at^{-1}-u_b}{u_a-u_b}.
$$
\end{lemma}

\begin{proof} The same argument as in Lemma \ref{lemma7.A}.
\end{proof}

\section{The Jack limit}\label{sect8}

In this section we are working over the base field $\Q(\kk)$, where $\kk$ is a new formal variable. 
Consider the symmetric Jack polynomials \cite[ch. VI, \S10]{M}. We denote them as $P_{\mu\mid N}(x_1,\dots,x_N;\kk)$, where $\mu\in\Y(N)$ as above. The connection with Macdonald's notation is the following: 
$$
P_{\mu\mid N}(x_1,\dots,x_N;\kk)=P^{(1/\kk)}_\mu(x_1,\dots,x_N), \qquad \mu\in\Y(N)
$$
(that is,  our $\kk$ is inverse to Macdonald's parameter $\al$). Recall that
\begin{equation}\label{eq8.A1}
P_{\mu\mid N}(x_1,\dots,x_N;\kk)=\lim_{q\to1}P_{\mu\mid N}(x_1,\dots,x_N;q,q^\kk);
\end{equation}
here and in what follows, whenever a limit transition $q\to1$ is considered, we temporarily regard  $q$ and $\kk$ as numeric parameters; but the final expression always makes sense over $\Q(\kk)$. 

The combinatorial presentation of the Jack polynomials can be written as
$$
P_{\mu\mid N}(x_1,\dots,x_N;\kk)=\sum_{T\in\RTab(\mu,N)}\psi_T(\kk) \prod_{(i,j)\in\mu}x_{T(i,j)},
$$
where $\psi_T(\kk)$ are certain rational functions of $\kk$; in Macdonald's notation,  $\psi_T(\kk)=\psi_T^{(1/\kk)}$, see \cite[ch. VI, (10.10), (10.11), and (10.12)]{M}. Again, these quantities can be obtained 
by a degeneration from $\psi_T(q,t)$:
$$
\psi_T(\kk):=\lim_{q\to1}\psi_T(q,q^\kk).
$$

Likewise, there are dual Jack polynomials
$$
Q_{\mu\mid N}(x_1,\dots,x_N;\kk)=Q^{(1/\kk)}_\mu(x_1,\dots,x_N), \qquad \mu\in\Y(N),
$$
with the combinatorial presentation
$$
Q_{\mu\mid N}(x_1,\dots,x_N;\kk)=\sum_{T\in\RTab(\mu,N)}\varphi_T(\kk) \prod_{(i,j)\in\mu}x_{T(i,j)},
$$
where $\varphi_T(\kk)=\varphi^{(1/\kk)}_T$ are rational functions of $\kk$ satisfying 
$$
\varphi_T(\kk):=\lim_{q\to1}\varphi_T(q,q^\kk).
$$
We also have
\begin{equation}\label{eq8.B1}
Q_{\mu\mid N}(x_1,\dots,x_N;\kk)=\lim_{q\to1}Q_{\mu\mid N}(x_1,\dots,x_N;q,q^\kk);
\end{equation}

We define now the Jack analogues of the interpolation Macdonald polynomials $I_{\mu\mid N}(\ccdot;q,t)$ and of the rational functions $\wt H_{\mu\mid N}(\ccdot;q,t)$  by means of combinatorial formulas: 
\begin{gather*}
I_{\mu\mid N}(x_1,\dots,x_N;\kk)=\sum_{T\in\RTab(\mu,N)}\psi_T(\kk)\prod_{(i,j)\in\mu}(x_{T(i,j)}-(j-1)+(T(i,j)+i-2)\kk),\\
\wt H_{\mu\mid N}(u_1,\dots,u_N;\kk)=\sum_{T\in\RTab(\mu,N)}\prod_{i=1}^N \Wm_{\mu(i-1)/\mu(i)}(u_i;\kk),
\end{gather*}
where $\mu=\mu(0)\succ\mu(1)\succ\dots\succ\mu(N)=\varnothing$ represents $T$ and,  for a horizontal strip $\la/\ka$,
$$
\Wm_{\la/\ka}(u;\kk):=\varphi_{\la/\ka}(\kk)\frac{\prod_{i=1}^{\ell(\la)}\prod_{j=1}^{\ka_i-\la_{i+1}}(u-(\ka_i+j-1)+i\kk)}{(u-1)\dots(u-\la_1)}.
$$

These formulas are obtained by degenerating the formulas \eqref{eq2.E},  \eqref{combJ}, and \eqref{eq5.B}: 
\begin{gather}
I_{\mu\mid N}(x_1,\dots,x_N;\kk):=\lim_{q\to1}(1-q)^{-|\mu|} I_{\mu\mid N}(q^{-x_1},\dots,q^{-x_N}; q, q^{\kk}), \label{eq8.A}\\
\wt H_{\mu\mid N}(u_1,\dots,u_N;\kk):=\lim_{q\to1}(1-q)^{|\mu|} \wt H_{\mu\mid N}(q^{-u_1},\dots,q^{-u_N}; q, q^{\kk}).\label{eq8.B}
\end{gather}

The polynomials $I_{\mu\mid N}(x_1,\dots,x_N;\kk)$ coincide (up to notation) with the polynomials introduced by Knop and Sahi in \cite{KnopSahi}. In the notation of Okounkov \cite{Ok-CM},
$$
I_{\mu\mid N}(x_1,\dots,x_N;\kk)=P^*_\mu(x_1,x_2+\kk,\dots,x_N+(N-1)\kk;\kk).
$$

The rational functions $\wt H_{\mu\mid N}(u_1,\dots,u_N;\kk)$ are alternately regarded as power series, that is, elements of  $\Q(\kk)[[u_1^{-1},\dots,u_N^{-1}]]$. 

Note that the degeneration in \eqref{eq8.A} and \eqref{eq8.B} is of different sort as compared with \eqref{eq8.A1} and \eqref{eq8.B1}. Nevertheless, we have 
\begin{gather}
I_{\mu\mid N}(x_1,\dots,x_N;\kk)=P_{\mu\mid N}(x_1,\dots,x_N;\kk) +\text{lower degree terms in $x_i$'s}, \label{eq8.C1}\\
\wt H_{\mu\mid N}(u_1,\dots,u_N;\kk)=Q_{\mu\mid N}(u^{-1}_1,\dots,u^{-1}_N;\kk) +\text{lower degree terms in $u_i$'s}.  \label{eq8.C2}
\end{gather}

Set
\begin{equation}\label{eq8.C3}
F(x,u;\kk):=1+\sum_{m=1}^\infty\frac{(\kk)_m}{m!}\frac{x(x-1)\dots(x-m+1)}{(u-1)\dots(u-m)}.
\end{equation}
We regard this series as an element of $\Q(\kk)[[x,u^{-1}]]$. 

\begin{theorem}\label{thm8.A}
The following Cauchy-type identity holds in the algebra of formal power series $\Q(\kk)[[x_1,\dots,x_N,u^{-1}_1,\dots,u^{-1}_N]]${\rm:}
\begin{equation}\label{eq8.D1}
\sum_{\mu\in\Y(N)}I_{\mu\mid N}(x_1,\dots,x_N;\kk)\wt H_{\mu\mid N}(u_1,\dots,u_N;\kk)=\prod_{i=1}^N\prod_{j=1}^N \frac{F(x_i, u_j; \kk)}{F(-(i-1)\kk,u_j;\kk)}. 
\end{equation}
\end{theorem}

\begin{remark}\label{rem8.A}
There is an alternative expression for the right-hand side of \eqref{eq8.D1}:
\begin{equation}\label{eq8.D2} 
\prod_{i=1}^N\prod_{j=1}^N \frac{F(x_i, u_j; \kk)}{F(-(i-1)\kk,u_j;\kk)}
=\prod_{i=1}^N\prod_{j=1}^N F(x_i+(i-1)\kk, u_j+(i-1)\kk; \kk).
\end{equation}

\end{remark}

\begin{proof}
\emph{Step 1: Application of difference operators.}
Our argument relies on Jack analogues of the $q$-difference operators $D_N(z;q,t)$ and $\wh D_N(z;q,t)$ from Section \ref{sect7}. These are certain difference operators, denoted as $D^\Jack_N(z;\kk)$ and $\wh D^\Jack_N(z,\kk)$. Let us describe their key properties (proofs will be given later) and explain how to use them for the derivation of \eqref{eq8.D1}. 

The operator $D^\Jack_N(z;\kk)$ depends polynomially on parameter $z$,  acts on the variables $x_1,\dots,x_N$, and one has (step 2 below)
\begin{equation}\label{eq8.E1}
D^\Jack_N(z;\kk) I_{\mu\mid N}(\ccdot;\kk)= \prod_{i=1}^N(\mu_i+(1-i)\kk+z) \cdot I_{\mu\mid N}(\ccdot;\kk), \qquad \mu\in\Y(N).
\end{equation}

Likewise, the operator $\wh D^\Jack_N(z;\kk)$ depends polynomially on parameter $z$,  acts on the variables $u_1,\dots,u_N$, and one has (step 3 below)
\begin{equation}\label{eq8.E2}
\wh D^\Jack_N(z;\kk) \wt H_{\mu\mid N}(\ccdot;\kk)=\prod_{i=1}^N(\mu_i+(1-i)\kk+z) \cdot \wt H_{\mu\mid N}(\ccdot;\kk), \qquad \mu\in\Y(N),
\end{equation}
with the same eigenvalues as in \eqref{eq8.E1}. 

Next, introduce a notation for the right-hand side of \eqref{eq8.D1}:
\begin{equation}\label{eq8.H}
\wt\Pi^\Jack(\ccdot;\kk)=\wt\Pi^\Jack(x_1,\dots,x_N,u_1,\dots,u_N;\kk):=\prod_{i=1}^N\prod_{j=1}^N \frac{F(x_i, u_j; \kk)}{F(-(i-1)\kk,u_j;\kk)}. 
\end{equation}
The following equality holds (steps 4--5 below)
\begin{equation}\label{eq8.E3}
D^\Jack_N(z;\kk)\wt\Pi^\Jack(\ccdot;\kk)\\
=\wh D^\Jack_N(z;\kk)\wt\Pi^\Jack(\ccdot;\kk),
\end{equation}
cf. \eqref{eq7.B}. 

Now we show how to deduce  \eqref{eq8.D1} from \eqref{eq8.E1}, \eqref{eq8.E2}, and \eqref{eq8.E3}.  

From the definition \eqref{eq8.C3} it follows that 
$\wt\Pi^\Jack(\ccdot;\kk)$ may be regarded as a symmetric formal power series in $u_1^{-1},\dots,u_N^{-1}$ with coefficients in the algebra of symmetric polynomials over $\Q(\kk)$. On the other hand, as is seen from \eqref{eq8.C2}, any symmetric formal power series in $u_1^{-1},\dots,u_N^{-1}$ is uniquely written as a (generally, infinite) linear combination of the elements $\wt H_{\mu\mid N}(\ccdot;\kk)$. Therefore, there exists a unique series expansion
\begin{equation}\label{eq8.E4}
\wt\Pi^\Jack(x_1,\dots,x_N,u_1,\dots,u_N;\kk)=\sum_{\mu\in\Y(N)}f_\mu(x_1,\dots,x_N) \wt H_{\mu\mid N}(u_1,\dots,u_N;\kk),
\end{equation}
where the coefficients $f_\mu(x_1,\dots,x_N)$ are some symmetric polynomials over $\Q(\kk)$. 

Applying to both sides of \eqref{eq8.E4} the operator $\wh D^\Jack_N(z;\kk)$ and using \eqref{eq8.E3} we obtain that the polynomials $f_\mu(x_1,\dots,x_N)$ are eigenfunctions of $D^\Jack_N(z;\kk)$ with the same eigenvalues as the polynomials $I_{\mu\mid N}(x_1,\dots,x_N;\kk)$. Since the eigenvalues (regarded as polynomials in $z$) are pairwise distinct, the two systems of polynomials are the same, up to proportionality. Therefore, \eqref{eq8.E4} can be written in the form
\begin{multline}\label{eq8.E5}
\wt\Pi^\Jack(x_1,\dots,x_N,u_1,\dots,u_N;\kk)\\
=\sum_{\mu\in\Y(N)}c_\mu I_{\mu\mid N}(x_1,\dots,x_N;\kk) \wt H_{\mu\mid N}(u_1,\dots,u_N;\kk)\end{multline}
with some coefficients $c_\mu\in\Q(\kk)$.

Now we are going to show that $c_\mu=1$ for all $\mu$. For this purpose, we multiply in \eqref{eq8.E5} all $x_i$ and $u_j$ by an extra parameter $r$. Observe that on the right-hand side of \eqref{eq8.E5} we have
\begin{multline}\label{eq8.E6}
I_{\mu\mid N}(x_1r,\dots,x_Nr;\kk) \wt H_{\mu\mid N}(u_1r,\dots,u_Nr;\kk)\\
=P_{\mu\mid N}(x_1,\dots,x_N;\kk) Q_{\mu\mid N}(u_1^{-1},\dots,u_N^{-1};\kk)+O(1/r),
\end{multline}
as it is seen from \eqref{eq8.C1} and \eqref{eq8.C2}.

Let us turn to the left-hand side of \eqref{eq8.E5}. From \eqref{eq8.C3} it follows that
\begin{equation*}
F(xr,ur;\kk)=1+\sum_{m=1}^\infty \frac{(\kk)_m}{m!}(xu^{-1})^m+O(1/r)=\frac1{(1-xu^{-1})^\kk}+O(1/r).
\end{equation*}
Note also that 
$$
F(\const,ur;\kk)=1+O(1/r).
$$
This implies  (see the definition \eqref{eq8.H})
\begin{equation}\label{eq8.E7}
\wt\Pi^\Jack(x_1r,\dots,x_Nr,u_1r,\dots,u_Nr;\kk)=\prod_{i=1}^N\prod_{j=1}^N\frac1{(1-x_iu_j^{-1})^\kk}+O(1/r).
\end{equation}

On the other hand, the Jack version of the Cauchy identity \cite[ch. VI, (10.4)]{M} tells us that  
$$
\sum_{\mu\in\Y(N)} P_{\mu\mid N}(x_1,\dots,x_N;\kk) Q_{\mu\mid N}(y_1,\dots,y_N;\kk)=\prod_{i=1}^N\prod_{j=1}^N\frac1{(1-x_i y_j)^\kk}.
$$
Comparing this with \eqref{eq8.E6} and \eqref{eq8.E7} we conclude that all coefficients $c_\mu$ are equal to $1$. 

To complete the proof we have to exhibit the difference operators with the desired properties. This is done below. 

\smallskip

\emph{Step 2: The difference operator $D^\Jack_N(z;\kk)$.} This operator is obtained from the operator $D_N(z;q,t)$ by a scaling limit transition similar to that in \eqref{eq8.A}. Namely, we regard $q$ and $\kk$ as  numeric parameters, replace each $x_i$ with $q^{-x_i}$, replace $z$ with $-q^z$, substitute $t=q^\kk$, then  divide the operator by $(1-q)^N$, and let $q$ go to $1$. To compute the limit it is convenient to deal not with the initial definition of $D_N(z;q,t)$ given in \eqref{eq7.A} but with the alternate expression from Lemma  \ref{lemma7.A}. Then we readily  obtain  
\begin{equation*}
D^\Jack_N(z;\kk)=\sum_A C^\Jack_A T^{-1}_{A,X},
\end{equation*}
where $A$ ranges, as before, over subsets of $\{1,\dots,N\}$,
\begin{equation*}
C^\Jack_A=C^\Jack_A(x_1,\dots,x_N;\kk):=\prod_{a\in A}(x_a+(N-1)\kk)\cdot\prod_{b\in A^c}(x_b+z) \cdot\prod_{a\in A}\prod_{b\in A^c} \frac{x_a-x_b-\kk}{x_a-x_b},
\end{equation*}
and $T_{A,X}:=\prod_{a\in A} T_{x_a}$, where the univariate operator $T_x$ denotes the additive shift,
$$
T_x f(x):=f(x+1).
$$
Next, from \eqref{eq8.A} and \eqref{eq7.A1} we obtain the desired equality \eqref{eq8.E1}. 

\smallskip

\emph{Step 3: The difference operator $\wh D^\Jack_N(z;\kk)$.} This operator is obtained from the operator $\wt D_N(z;q,t)$ in exactly the same way. The only difference is that we are dealing with variables $u_i$ (which are replaced with $q^{-u_i}$) and use Lemma \ref{lemma7.B} instead of Lemma \ref{lemma7.A}. The resulting operator has the form  
\begin{equation*}
\wh D^\Jack_N(z;\kk)=\sum_A \wh C^\Jack_A T_{A,U}
\end{equation*}
where 
\begin{equation*}
\wh C^\Jack_A=\wh C^\Jack_A(u_1,\dots,u_N;\kk):=\prod_{a\in A}u_a\cdot\prod_{b\in A^c}(u_b+z) \cdot\prod_{a\in A}\prod_{b\in A^c} \frac{u_a-u_b+\kk}{u_a-u_b},
\end{equation*}
and $T_{A,U}:=\prod_{a\in A}T_{u_a}$. Next, from \eqref{eq8.B} and \eqref{eq7.A2} we obtain the desired equality \eqref{eq8.E2}. 

\smallskip

\emph{Step 4: Proof of identity \eqref{eq8.E3} {\rm(}beginning{\rm)}.}  This identity can be written as  
\begin{equation}\label{eq8.G}
\sum_A C^\Jack_A\frac{T^{-1}_{A,X}\wt\Pi^\Jack(\ccdot;\kk)}{\wt\Pi^\Jack(\ccdot;\kk)}
= \sum_A \wh C^\Jack_A\frac{T_{A,U}\wt\Pi^\Jack(\ccdot;\kk)}{\wt\Pi^\Jack(\ccdot;\kk)}.
\end{equation}
We are going to derive \eqref{eq8.G} from the relation \eqref{eq7.B}. We may rewrite \eqref{eq7.B} as
\begin{equation}\label{eq8.G1}
\sum_A C_A \frac{T_{q,A,X}\wt\Pi(\ccdot;q,t)}{\wt\Pi(\ccdot;q,t)}
= \sum_A \wh C_A\frac{T^{-1}_{q,A,U}\wt\Pi(\ccdot;q,t)}{\wt\Pi(\ccdot;q,t)},
\end{equation}
where $\wt\Pi(\ccdot;q,t)$ is defined by \eqref{eq7.Pi}.
Observe that
\begin{equation}\label{eq8.G2}
\frac{T_{q,A,X}\wt\Pi(\ccdot;q,t)}{\wt\Pi(\ccdot;q,t)}=\prod_{a\in A}\prod_{j=1}^N\frac{1-x_au_j^{-1}}{1-x_au_j^{-1}t}:=E_A
\end{equation}
and
\begin{equation}\label{eq8.G3}
\frac{T^{-1}_{q,A,U}\wt\Pi(\ccdot;q,t)}{\wt\Pi(\ccdot;q,t)}=\prod_{i=1}^N\prod_{a\in A}\frac{1-x_iu_a^{-1}}{1-x_i u_a^{-1}t} \cdot\prod_{a\in A}\frac{1-u_a^{-1}t^N}{1-u_a^{-1}}:=\wh E_A.
\end{equation}
In this notation, \eqref{eq8.G1} takes the form
\begin{equation}\label{eq8.G4}
\sum_A C_A E_A=\sum_A \wh C_A \wh E_A,
\end{equation}
which is an identity of rational functions in variables $x_i$'s, $u_j$'s, $z$, and $t$. 

We know that in our scaling limit regime,  $(1-q)^{-N}C_A$ and $(1-q)^{-N}\wh C_A$ tend to  $C^\Jack_A$ and $\wh C^\Jack_A$, respectively. Next, in the same limit regime, $E_A$ and $\wh E_A$ are transformed into
$$
\prod_{a\in A}\prod_{j=1}^N\frac{u_j-x_a}{u_j-x_a+\kk} \quad \text{and} \quad \prod_{i=1}^N\prod_{a\in A}\frac{u_a-x_i}{u_a-x_i+\kk}\cdot \prod_{a\in A}\frac{u_a+N\kk}{u_a},
$$
respectively.

Therefore, to complete the proof of \eqref{eq8.E3} it suffices to check the following two relations: for any $A\subseteq\{1,\dots,N\}$
\begin{equation}\label{eq8.G5}
\frac{T^{-1}_{A,X}\wt\Pi^\Jack(\ccdot;\kk)}{\wt\Pi^\Jack(\ccdot;\kk)}
=\prod_{a\in A}\prod_{j=1}^N\frac{u_j-x_a}{u_j-x_a+\kk}
\end{equation}
and
\begin{equation}\label{eq8.G6}
\frac{T_{A,U}\wt\Pi^\Jack(\ccdot;\kk)}{\wt\Pi^\Jack(\ccdot;\kk)}=\prod_{i=1}^N\prod_{a\in A}\frac{u_a-x_i}{u_a-x_i+\kk}\cdot \prod_{a\in A}\frac{u_a+N\kk}{u_a}.
\end{equation}

\emph{Step 5: Proof of identity \eqref{eq8.E3} {\rm(}end\/{\rm)}.} We proceed to the proof of \eqref{eq8.G5} and \eqref{eq8.G6}. Recall that $\wt\Pi^\Jack(\ccdot;\kk)$ is defined by \eqref{eq8.H}, so both \eqref{eq8.G5} and \eqref{eq8.G6} express some properties of the series $F(x,u;\kk)$. We claim that the following relations hold
\begin{gather}
T_x^{-1} F(x,u;\kk)=\frac{u-x}{u-x+\kk} F(x,u;\kk),   \label{eq8.H1}\\
T_u F(x,u;\kk)=\frac{(u-x)(u+\kk)}{(u-x+\kk)u}F(x,u;\kk). \label{eq8.H2}
\end{gather}
Observe that the desired relations \eqref{eq8.G5} and \eqref{eq8.G6} are readily deduced from \eqref{eq8.H1}, \eqref{eq8.H2}, and the definition \eqref{eq8.H}. Thus, it remains to check \eqref{eq8.H1} and  \eqref{eq8.H2}. 

As is seen from the definition \eqref{eq8.C3}, the formal series  $F(x,u;\kk)$ is nothing else than the Gauss hypergeometric series at the point $1$:
$$
F(x,u;\kk)={}_2F_1(\kk,-x;-u+1;1).
$$
We would like to apply the Gauss formula
\begin{equation*}
{}_2F_1(\kk,-x;-u+1;1)=\frac{\Ga(-u+1)\Ga(x-u+1-\kk)}{\Ga(x-u+1)\Ga(-u+1-\kk)};
\end{equation*}
then the relations \eqref{eq8.H1} and \eqref{eq8.H2} would immediately follow from it and the functional equation for the gamma function. 

A subtle point is that our formal series in $(\kk, x, u^{-1})$ absolutely converges  in the domain 
$$
\{(\kk, x,u): \Re u<\Re(x-\kk+1), \; u\notin\Z_{\ge1}\}\subset\C^3,
$$
which does not cover the set
$$
\{(\kk,x,u): |\kk|<\eps, \;|x|<\eps,\; |u|>\eps^{-1}\}\subset \C^3,
$$
no matter how small $\eps$ is. This does not allow us to apply the Gauss formula directly. 

However, one can overcome this obstacle in the following way. The series $F(x,u;\kk)$ lies in fact in the algebra $(\Q[x,\kk])[[u^{-1}]]$ (formal series in $u^{-1}$ with coefficients in the polynomial algebra $\Q[x,\kk]$). Consequently it suffices to prove the relations \eqref{eq8.H1} and \eqref{eq8.H2} under assumption that $x$ is specialized to $0,1,2,\dots$\,. But then the series terminates and represents a rational function, and then the use of Gauss' formula is justified. 

This completes the proof of the theorem. 
\end{proof}

\begin{remark}
From the proof of Theorem \ref{thm8.A} one can deduce the identity
\begin{multline}\label{eq8.Gamma}
\sum_{\mu\in\Y(N)}I_{\mu\mid N}(x_1,\dots,x_N;\kk)\wt H_{\mu\mid N}(u_1,\dots,u_N;\kk)\\
=\prod_{i=1}^N\prod_{j=1}^N\frac{\Ga(x_i-u_j-\kk+1)}{\Ga(x_i-u_j+1)}
\cdot\prod_{j=1}^N\frac{\Ga(-u_j+1)}{\Ga(-u_j-N\kk+1)},
\end{multline}
where, in contrast to Theorem \ref{thm8.A},  both sides are treated not as formal series but as functions of complex variables, while $\kk$ is supposed to be a positive real number; the series on the left absolutely converges provided that $\max_j(\Re u_j)\ll0$ (depending on $x_1,\dots,x_N$). Conversely, Theorem \ref{thm8.A} can be deduced from this identity. 

One can derive \eqref{eq8.Gamma} from the $(q,t)$ Cauchy identity \eqref{eq3.B} by a limit transition, which gives an alternate proof of Theorem \ref{thm8.A}. The scheme of proof is the following.

1. We take $q>1$ and $t=q^\kk$ with $\kk>0$, and show that the series on the left-hand side of \eqref{eq3.B} absolutely converges provided $\max_j |u_j|\ll 1$ (depending on $\max_i |x_i|$). 

2. On the right-hand side, we pass from $(\ccdot;q)_\infty$ to $(\ccdot;q^{-1})_\infty$ according to \eqref{inversion}. Then we rewrite the whole expression in terms of the $q^{-1}$-Gamma function $\Ga_{q^{-1}}(\ccdot)$ (recall that $q>1$ so that $0<q^{-1}<1$). 

3. We make the change of variables as in \eqref{eq8.A} and \eqref{eq8.B}, and we show that the summands on the left are nonnegative provided that $x_i$'s and $u_j$'s are real and such that $u_j\ll  x_i\ll 0$ for all $i$ and $j$. This enables us to pass to the limit as $q\to 1^+$.  Here we use the fact that $\Ga_q(z)\to \Ga(z)$ as $q\to1^-$. This leads us to \eqref{eq8.Gamma} for real arguments subject to constraints $u_j\ll  x_i\ll 0$.

4. We pass to complex values of arguments using analytic continuation. 

However, an accurate exposition of this argument is a bit tedious and would require no less space than the proof given above.  

Finally, note that without the second product on the right-hand side of \eqref{eq3.B} the limit transition would be impossible. Recall that this extra product arose thanks to the passage to modified dual functions (see the end of Section \ref{sect3}).

\end{remark}

\section{Specializations at $t=0$ and $q=0$}\label{sect9}

\subsection{The case $t=0$}

From the definition of the factors $\psi_T(q,t)$ attached to semistandard tableaux it follows that these factors can be specialized at $t=0$, so that the quantities $\psi_T(q,0)\in\Q(q)$ are well defined. This implies that the Macdonald polynomials $P_{\mu\mid N}(x_1,\dots,x_N;q,t)$ can be specialized at $t=0$, too. The resulting polynomials $P_{\mu\mid N}(x_1,\dots,x_N;q,0)$ over $\Q(q)$ are called the \emph{$q$-Whittaker polynomials}. 

The same holds for the interpolation Maconald polynomials, as is seen from the combinatorial formula \eqref{eq2.E}. We set 
$$
\A^\W_{\mu\mid N}(x_1,\dots,x_N;q):=I_{\mu\mid N}(x_1,\dots,x_N;q,0), \qquad \mu\in\Y(N). 
$$
Evidently, 
$$
\A^\W_{\mu\mid N}(x_1,\dots,x_N;q)=P_\mu(x_1,\dots,x_N;q,0)+\text{lower degree terms}.
$$ 
From \eqref{eq2.E}  we obtain 
$$
\A^\W_{\mu\mid N}(x_1,\dots,x_N;q)=\sum_{T\in\RTab(\mu,N)} \psi_T(q,0)\prod_{(i,j)\in\mu}(x_{T(i,j)}-q^{1-j}\eps_T(i,j)),
$$
where
$$
\eps_T(i,j)=\begin{cases}1, & \text{\rm if $i=1$ and $T(1,j)=1$,} \\ 0, & \text{\rm otherwise} \end{cases}.
$$

Likewise, the combinatorial formula for the modified dual functions (Corollary \ref{cor5.A}) shows that these functions also can be specialized at $t=0$. We set (below $K$ is a positive integer)
$$
\B^\W_{\mu\mid K}(u_1,\dots,u_K;q):=\wt H_{\mu\mid K}(u_1,\dots,u_K;q,0), \qquad \mu\in\Y(K). 
$$
These are rational functions over $\Q(q)$. The formula of Corollary \ref{cor5.A} specializes to
$$
\B^\W_{\mu\mid K}(u_1,\dots,u_K;q)=\sum_{T\in\RTab(\mu,K)}\prod_{i=1}^K \Wm_{\mu(i-1)/\mu(i)}(u_i;q),
$$
where, for a horizontal strip $\la/\ka$, 
$$
\Wm_{\la/\ka}(u;q):=\varphi_{\la/\ka}(q,0)\frac{u^{|\ka|-|\la|+\la_1}}{(u-q^{-1})\dots(u-q^{-\la_1})}.
$$

\begin{proposition}
The following Cauchy-type identity holds
\begin{multline}\label{eq9.A1}
\sum_{\mu:\, \ell(\mu)\le\min(N,K)}\A^W_{\mu\mid N}(x_1,\dots,x_N;q)\B^W_{\mu\mid K}(u_1,\dots,u_K;q)\\
=\prod_{i=1}^N\prod_{j=1}^K\frac1{(x_iu^{-1}_j;q)_\infty}\cdot\prod_{j=1}^K (u^{-1}_j;q)_\infty.
\end{multline}
Here both sides should be thought as elements of the $\Q(q)$-algebra of formal power series in $x_1,\dots,x_N, u^{-1}_1,\dots,u^{-1}_K$.
\end{proposition}

\begin{proof}
This is  as a direct consequence of  \eqref{eq3.B}.
\end{proof}

Note that the polynomials $\A^W_{\mu\mid N}(x_1,\dots,x_N;q)$ are stable in the conventional sense, which makes it possible to define symmetric functions  $\A^W_{\mu}(x_1,x_2,\dots;q)$ over the base field $\Q(q)$. The identity \eqref{eq9.A1} can be extended to the case of infinitely many variables.

\subsection{The case $q=0$}

\begin{lemma}\label{lemma9.B1}
Let $N$ be a positive integer and $\mu\in\Y(N)$.

{\rm(i)}The polynomials 
$$
\A^\HL_{\mu\mid N}(x_1,\dots,x_N;t):=I_{\mu\mid N}(x_1,\dots,x_N;q^{-1},t^{-1})\big|_{q=0}
$$
are well defined as elements of the algebra\/ $\Sym(N)\otimes_\Q\Q(t)$.

{\rm(ii)} The top degree homogeneous component of the polynomial $\A^\HL_{\mu\mid N}(x_1,\dots,x_N;t)$ is the Hall--Littlewood polynomial $P_\mu(x_1,\dots,x_N;0,t)$. 

{\rm(iii)} The following combinatorial formula holds{\rm:}
$$
\A^\HL_\mu(x_1,\dots,x_N;t)=\sum_{T\in\RTab(\mu,N)} \psi_T(0,t)\prod_{(i,j)\in\mu}(x_{T(i,j)}-t^{2-T(i,j)-i}\one_{j=1}),
$$
where $\one_{j=1}$ equals $1$ if $j=1$, and $0$ otherwise. 
\end{lemma}

Note that this formula is very close to the combinatorial formula for the Hall--Littlewood polynomials, the only difference comes from the boxes of the first column of the diagram $\mu$. 

\begin{proof}
As is well known, the Macdonald polynomials $P_{\mu\mid N}(x_1,\dots,N;q,t)$ are invariant under the change $(q,t)\to(q^{-1},t^{-1})$ (see \cite[ch. VI, (4.14) (iv)]{M}). This implies that $\psi_T(q^{1},t^{-1})=\psi_T(q,t)$, which can also be deduced from the explicit formula \cite[ch. VI, (6.24) (ii)]{M}. 
It follows, in particular, that 
$$
P_{\mu\mid N}(x_1,\dots,x_N;q^{-1},t^{-1})\big|_{q=0}=P_{\mu\mid N}(x_1,\dots,x_N;0,t).
$$

Using these facts and applying the combinatorial formula \eqref{eq2.E} one immediately obtains all the claims of the lemma. 
\end{proof}

\begin{lemma}\label{lemma9.B2}
Let $K$ be a positive integer and $\mu\in\Y(K)$. 

{\rm(i)} The specialization  
$$
\B^\HL_{\mu\mid K}(y_1,\dots,y_N;t):= \Idm_{\mu\mid K}(u_1t^{-1}q,\dots,u_K t^{-1}q;q^{-1},t^{-1})\big|_{q=0}
$$
makes sense and gives symmetric rational functions over\/ $\Q(t)$.  

{\rm(ii)} One has 
$$
\B^\HL_{\mu\mid K}(u_1,\dots,u_N;t)=Q_{\mu\mid K}(u^{-1}_1,\dots,u^{-1}_K;t) +\text{\rm higher degree terms in $u^{-1}_1,\dots,u^{-1}_K$},
$$
where the rational function on the left is treated as an element of\/ $\Q(t)[[u^{-1}_1,\dots,u^{-1}_K]]$ and $Q_{\mu\mid K}(\ccdot;t)$ is the dual Hall--Littlewood polynomial. 

{\rm(iii)} The rational functions $\B^\HL_{\mu\mid K}(y_1,\dots,y_N;t)$ are given by the combinatorial formula resulting from the branching rule 
$$
\B^\HL_{\mu\mid K}(u_1,\dots,u_K;t)=\sum_{\nu:\, \nu\prec\mu} W_{\mu/\nu}(u_1;t)\B^\HL_{\nu\mid K-1}(u_2,\dots,u_K;t),
$$
where 
$$
W_{\mu/\nu}(u;t):=\begin{cases} \varphi_{\mu/\nu}(0,t)\,\dfrac{u^{|\nu|-|\mu|+1}}{u-t}, & \ell(\nu)<\ell(\mu), \\
\varphi_{\mu/\nu}(0,t)\,\dfrac{u-t^{1-\ell}}{u-t} u^{|\nu|-|\mu|}, & \ell(\nu)=\ell(\mu):=\ell.
\end{cases}
$$
\end{lemma}

\begin{proof}
All the claims follow from the examination of the combinatorial formula for the functions $\Idm_{\mu\mid K}(u_1,\dots,u_K;q,t)$, see Corollary \ref{cor5.A}. By \eqref{eq5.B}, the weight attached to a horizontal strip $\mu/\nu$ is 
\begin{multline*}
\Wm_{\mu/\nu}(u;q,t):=\varphi_{\mu/\nu}(q,t)\dfrac{\prod\limits_{i=1}^{\ell(\mu)}\prod\limits_{j=1}^{\nu_i-\mu_{i+1}}(u-q^{-\nu_i+j-1}t^i)}{(u-q^{-1})\dots(u-q^{-\mu_1})}\\
=\varphi_{\mu/\nu}(q,t)\dfrac{\prod\limits_{i=1}^{\ell(\mu)}\prod\limits_{m=\mu_{i+1}+1}^{\nu_i}(u-q^{-m} t^i)}{(u-q^{-1})\dots(u-q^{-\mu_1})},
\end{multline*}
with the understanding that the interior product over $m$ equals $1$ if $\nu_i=\mu_{i+1}$.

Note that 
$$
\varphi_{\mu/\nu}(q^{-1},t^{-1})=\left(\frac qt\right)^{|\mu|-|\nu|}\varphi_{\mu/\nu}(q,t),
$$
as it seen from \cite[ch. VI, (6.24) (i)]{M}. Using this we obtain
$$
\Wm_{\mu/\nu}(uq t^{-1};q^{-1},t^{-1}):=\varphi_{\mu/\nu}(q,t)\dfrac{\prod\limits_{i=1}^{\ell(\mu)}\prod\limits_{m=\mu_{i+1}}^{\nu_i-1}(u-q^{m} t^{1-i})}{(u-t)(u-tq)\dots(u-tq^{\mu_1-1})}.
$$
Here we assume $\mu_1\ge1$, otherwise there is nothing to prove. 

Now it is clear that the resulting expression may be specialized at $q=0$. The denominator turns into $(u-t)u^{\mu_1-1}$. In the numerator, after setting $q=0$, each factor becomes equal to $u$ unless $\mu_{i+1}=0$ and $\nu_i>\mu_{i+1}$. These two conditions just mean that $i=\ell:=\ell(\mu)$ and $\ell(\nu)=\ell$, and then the only exceptional factor is $(u-t^{1-\ell})$. 

This gives the desired weight $W_{\mu/\nu}(u;t)$ and completes the proof. 
\end{proof}

\begin{proposition}\label{prop9.B}  
\emph{The following Cauchy--type identity holds}
\begin{multline}\label{limitCauchy}
\sum_{\mu: \,\ell(\mu)\le \min(N,K)}\A^\HL_{\mu\mid N}(x_1,\dots,x_N;t) \B^\HL_{\mu\mid K}(u_1,\dots,u_K;t)\\
=\prod_{i=1}^N\prod_{j=1}^K\frac{u_j-x_it}{u_j-x_i}\cdot\prod_{j=1}^K\frac{u_j-t^{1-N}}{u_j-t}.
\end{multline}
\end{proposition}

\begin{proof}
By virtue of \eqref{eq3.B} and \eqref{inversion} 
\begin{multline*}
\sum_{\mu: \,\ell(\mu)\le \min(N,K)}I_{\mu\mid N}(x_1,\dots,x_N;q^{-1},t^{-1})\Idm_{\mu\mid K}(u_1,\dots,u_K;q^{-1},t^{-1})\\
=\prod_{i=1}^N\prod_{j=1}^K\frac{(x_iu^{-1}_jt^{-1};q^{-1})_\infty}{(x_iu^{-1}_j;q^{-1})_\infty}\cdot\prod_{j=1}^K\frac{(u^{-1}_j;q^{-1})_\infty}{(u^{-1}_jt^{-N};q^{-1})_\infty}\\
=\prod_{i=1}^N\prod_{j=1}^K\frac{(x_iu^{-1}_j q;q)_\infty}{(x_iu^{-1}_j t^{-1}q;q)_\infty}\cdot\prod_{j=1}^K\frac{(u^{-1}_j t^{-N}q;q)_\infty}{(u^{-1}_j q;q)_\infty}.
\end{multline*}
Next, replacing each $u_j$ with $u_jqt^{-1}$ we obtain
\begin{multline*}
\sum_{\mu: \,\ell(\mu)\le \min(N,K)}I_{\mu\mid N}(x_1,\dots,x_N;q^{-1},t^{-1})\Idm_{\mu\mid K}(u_1qt^{-1},\dots,u_K qt^{-1};q^{-1},t^{-1})\\
=\prod_{i=1}^N\prod_{j=1}^K\frac{(x_iu^{-1}_j t;q)_\infty}{(x_iu^{-1}_j;q)_\infty}\cdot\prod_{j=1}^K\frac{(u^{-1}_j t^{1-N};q)_\infty}{(u^{-1}_j t;q)_\infty}.
\end{multline*}
Finally, this identity may be specialized at $q=0$. For the left-hand side we use Lemmas \ref{lemma9.B1} and \ref{lemma9.B2}, and this gives the left-hand side of \eqref{limitCauchy}. For the right-hand side this directly gives the right-hand side of \eqref{limitCauchy}.
\end{proof}

Note that the $\B^\HL$-functions are stable,
$$
\B^\HL_{\mu\mid K}(y_1,\dots,y_{K-1}, 0;t)=\B^\HL_{\mu\mid K-1}(y_1,\dots,y_{K-1};t),
$$
while the $\A^\HL$-polynomials are only quasi-stable in the sense that
$$
\A^HL_{\mu\mid N}(x_1,\dots,y_{N-1}, t^{1-N};t)=\A^\HL_{\mu\mid N-1}(x_1,\dots,x_{N-1};t).
$$

This agrees with the structure of \eqref{limitCauchy}. The quasi-stability reflects in the fact that detaching one variable in the $\A^\HL$-polynomial requires a shift of the remaining variables:
$$
\A^\HL_{\mu\mid N}(x_1,\dots,x_N;t)=\sum_{\nu:\, \nu\prec\mu} (\text{a weight in variable $x_1$})\cdot \A^\HL_{\nu\mid N-1}(x_2t^{-1},\dots,x_N t^{-1};t).
$$

\end{document}